\documentclass[11pt,a4paper,reqno]{amsart}

\usepackage{amsmath,amssymb,amsthm,mathtools,amscd,xcolor}
\usepackage{array,booktabs,enumitem,microtype}
\usepackage{tikz-cd}
\usepackage{hyperref}

\newcommand{\Aut}{\operatorname{Aut}}
\newcommand{\BCay}{\operatorname{BCay}}

\newcommand{\CT}{\operatorname{CT}}
\newcommand{\Syl}{\operatorname{Syl}}
\newcommand{\GL}{\operatorname{GL}}

\newcommand{\Beg}{\operatorname{beg}}
\newcommand{\End}{\operatorname{end}}
\newcommand{\Inv}{\operatorname{inv}}
\newcommand{\R}{\mathrm R}
\newcommand{\F}{\mathbb F}
\newcommand{\Z}{\mathbb Z}

\newcommand{\id}{\operatorname{id}}

\newcommand{\Q}{\mathbb{Q}}
\newcommand{\rank}{\mathrm{rank}}

\newcommand{\diag}{\mathrm{diag}}
\newcommand{\Cay}{\mathrm{Cay}}

\newtheorem{theorem}{Theorem}[section]
\newtheorem{proposition}[theorem]{Proposition}
\newtheorem{lemma}[theorem]{Lemma}
\newtheorem{corollary}[theorem]{Corollary}
\theoremstyle{definition}
\newtheorem{definition}[theorem]{Definition}
\newtheorem{remark}[theorem]{Remark}
\newtheorem{notation}[theorem]{Notation}

\begin{document}

\title[Generalized quaternion groups of order $4p$ are $4$-BCI-groups]{Generalized Quaternion Groups of Order \texorpdfstring{$4\lowercase{p}$}{4p} Are \texorpdfstring{$4$}{4}-BCI-Groups}

\author[J.~Xie]{Jin-Hua Xie}
\address{School of Mathematics and Statistics, Hainan Normal University, Haikou, P.~R.~China}
\email{jinhuaxie@hainnu.edu.cn}

\begin{abstract}

For a finite group $G$ and a subset $S\subseteq G$, the bi-Cayley graph $\BCay(G,S)$ has vertex set $G\times\{0,1\}$ and edges $\{(x,0),(sx,1)\}$ for $x\in G$ and $s\in S$. The group $G$ is called an $m$-BCI-group if, for all subsets $S,T\subseteq G$ with $|S|\leq m$, the isomorphism $\BCay(G,S)\cong\BCay(G,T)$ implies that $T=gS^\alpha$ for some $g\in G$ and $\alpha\in\Aut G$. We prove that, for every prime $p$, the generalized quaternion group $Q_{4p}$ is a $4$-BCI-group. For $p\geq5$, the proof uses regular coverings and voltage assignments: we pass to quotient graphs modulo the Sylow $p$-subgroup and lift isomorphisms between them.



\end{abstract}

\maketitle

\section{Introduction}
The isomorphism problem for Cayley graphs is a classical topic in algebraic graph theory. Let $G$ be a finite group and let $S\subseteq G$ with $S=S^{-1}$. The Cayley graph $\Cay(G,S)$ is the graph with vertex set $G$ and edge set $\bigl\{\{x,sx\}\mid x\in G,\ s\in S\bigr\}$. A Cayley graph $\Cay(G,S)$ is called a \emph{CI-graph} if, for every $T\subseteq G$ with $T=T^{-1}$, the isomorphism
\[
    \Cay(G,S)\cong\Cay(G,T)
\]
implies that $T=S^\alpha$ for some $\alpha\in\Aut G$. A group $G$ is called a \emph{CI-group} if every Cayley graph over $G$ is a CI-graph. The CI-problem asks which Cayley graphs are CI-graphs and which groups are CI-groups. The problem goes back to the work of Babai~\cite{Babai1977} and has been studied extensively (see, for example, the survey~\cite{Li2002}).

A related isomorphism problem arises for bi-Cayley graphs. For a finite group $G$ and a subset $S\subseteq G$, the bi-Cayley graph $\BCay(G,S)$ has vertex set $G\times\{0,1\}$ and edge set
\[
\bigl\{\{(x,0),(sx,1)\}\mid x\in G,\ s\in S\bigr\}.
\]
The above construction is also known as a Haar graph in the literature (see, for example,~\cite{JinLiu2010,KoikeKovacs2014}). Two subsets $S,T\subseteq G$ are called \emph{BCI-equivalent} if
\[
T=gS^\alpha
\]
for some $g\in G$ and $\alpha\in\Aut G$. It is well known that, if $S$ and $T$ are BCI-equivalent, then $\BCay(G,S)\cong\BCay(G,T)$ (see, for example,~\cite[Lemma~2.2]{JinLiu2010}). A subset $S\subseteq G$ is called a \emph{BCI-subset} if, for every $T\subseteq G$, the isomorphism
\[
\BCay(G,S)\cong\BCay(G,T)
\]
implies that $S$ and $T$ are BCI-equivalent. The group $G$ is called an \emph{$m$-BCI-group} if every subset of $G$ of cardinality at most $m$ is a BCI-subset.

The BCI-property was introduced in~\cite{XuJinEtAl2008}, where the $2$-BCI-groups were characterized. Subsequent work has considered the $3$-BCI problem for several classes of groups. Jin and Liu~\cite{JinLiu2010} proved that $A_5$ is the only nonabelian simple $3$-BCI-group, while Koike and Kov\'acs~\cite{KoikeKovacs2019} classified the nilpotent $3$-BCI-groups. For connection sets of size $4$, Koike and Kov\'acs~\cite{KoikeKovacs2014} determined when two connected tetravalent bi-Cayley graphs over a cyclic group are isomorphic, and isomorphisms of bi-Cayley graphs over dihedral groups were studied in~\cite{ABGR2020}.

Generalized quaternion groups form a further important class of groups in the study of the BCI-problem. Related isomorphism problems for Cayley graphs over these groups have been studied in~\cite{Ma1997,XFK2022}. More recently, the BCI-property of generalized quaternion groups was investigated in~\cite{XieZhang2026}, where it was proved that, for
\[
Q_{4n}=\langle a,b\mid a^{2n}=1,\ b^2=a^n,\ b^{-1}ab=a^{-1}\rangle
\]
and $m\in\{2,3\}$, the group $Q_{4n}$ is an $m$-BCI-group if and only if $n$ is odd or $n=2$. The groups $Q_{4n}$ are also called dicyclic groups; some authors reserve the name generalized quaternion group for the case where $n$ is a power of $2$.

The present paper treats the next case of the BCI-problem for generalized quaternion groups. Our main result is the following.

\begin{theorem}\label{thm:intro-Q4p-4BCI}
For every prime $p$, the generalized quaternion group $Q_{4p}$ is a $4$-BCI-group.
\end{theorem}

The case $p=2$ follows from~\cite[Lemma~3.4]{JinLiu2013}. For odd $p$, the $3$-BCI-property of $Q_{4p}$ proved in~\cite{XieZhang2026} reduces the theorem to connection sets of size $4$, and connection sets that contain $1$ and are contained in $\langle a\rangle$ are handled by known results on tetravalent bi-Cayley graphs over cyclic groups~\cite{ABGR2020,KoikeKovacs2014}. The remaining connection sets require a new argument.

Our main tool is the theory of regular coverings and voltage assignments. Every regular covering of a graph is described by a voltage assignment on the base graph~\cite{GT1977}, and whether an automorphism of the base graph lifts along a regular covering is determined by the voltages of closed walks~\cite{MNS2000}. For elementary abelian coverings, this becomes a linear condition on first homology over a finite field~\cite{MMP2004}. These tools are widely used in algebraic graph theory, for example to construct and classify symmetric graphs as regular coverings of small graphs and to lift prescribed groups of automorphisms; see~\cite{ConderMa2013,FengKwak2004,MMP2004,PS2019}. To our knowledge, however, they have not been used for CI- or BCI-problems, in which two different graphs must be compared rather than automorphisms of a single graph lifted.

Although these tools are available, the key question is how to apply them when two graphs $\Gamma$ and $\Gamma'$ are compared, that is, how to take proper quotients of $\Gamma$ and $\Gamma'$. We take quotients by semiregular subgroups $P\leq\Aut\Gamma$ and $P'\leq\Aut\Gamma'$. Since we are given only the two graphs and no isomorphism between them, $P$ and $P'$ must be chosen consistently: if $\Gamma\cong\Gamma'$, then some isomorphism $f\colon\Gamma\to\Gamma'$ must satisfy $f^{-1}Pf=P'$. Such an $f$ maps the $P$-orbits onto the $P'$-orbits and induces an isomorphism of the quotient graphs (Lemma~\ref{lem:quotient-alignment}). Without a consistent choice, $\Gamma$ and $\Gamma'$ may be isomorphic although no isomorphism between their quotient graphs lifts. This leads to the following reduction.

\begin{enumerate}[label=\textsf{Step \arabic*.},ref=\arabic*]

\item\label{step:intro-choose-quotients}
Choose semiregular subgroups $P\leq\Aut\Gamma$ and $P'\leq\Aut\Gamma'$ consistently, that is, such that, whenever $\Gamma\cong\Gamma'$, there exists an isomorphism $f\colon\Gamma\to\Gamma'$ with $f^{-1}Pf=P'$.

\item\label{step:intro-quotient-isomorphisms}
Determine the isomorphisms between the quotient graphs $\Gamma_P$ and $\Gamma'_{P'}$.

\item\label{step:intro-lifting}
Determine which such quotient isomorphisms lift to isomorphisms between $\Gamma$ and $\Gamma'$.

\end{enumerate}

For connected $d$-valent $t$-Cayley graphs over a group $G$ (Definition~\ref{def:t-Cayley}), a consistent choice is given by $P=\R_\Gamma(N)$ and $P'=\R_{\Gamma'}(N)$ for a Sylow $p$-subgroup $N$ of $G$, for every prime $p>d$ not dividing $t$ (Proposition~\ref{prop:sylow-alignment}). When $P$ and $P'$ are elementary abelian $p$-groups, Step~\ref{step:intro-lifting} can be expressed in terms of the maps induced by the voltage assignments on first homology over $\F_p$ (Corollary~\ref{cor:homology-comparison}, which reformulates~\cite[Corollary~3.3]{MMP2004} for two different base graphs). For fixed quotient graphs and fixed voltage assignments, Step~\ref{step:intro-lifting} is therefore a finite computation over $\F_p$ (Remark~\ref{rem:computational-lifting}).

For tetravalent bi-Cayley graphs over $Q_{4p}$ with $p\geq5$, we take $P$ and $P'$ to be the images of the Sylow $p$-subgroup of $Q_{4p}$. Only eight quotient graphs occur, and the voltages depend linearly on the parameters of the connection sets. This allows us to carry out Steps~\ref{step:intro-quotient-isomorphisms} and~\ref{step:intro-lifting} for all these graphs and all primes $p\geq5$ at once, by a computation with finitely many integer matrices (Section~\ref{sec:Q4p}). The case $p=3$ is settled by a direct computation.

The paper is organized as follows. Section~\ref{sec:preliminaries} introduces the preliminaries on graph coverings, voltage assignments, and homology needed later. Section~\ref{sec:reduction} develops the quotient reduction and the homological lifting criterion. In Section~\ref{sec:Q4p}, we apply the method to tetravalent bi-Cayley graphs over $Q_{4p}$ and prove Theorem~$\ref{thm:intro-Q4p-4BCI}$. Appendix~\ref{apdx:computational-verification} records the computational details used in the application.

\section{Graph coverings and voltage assignments}\label{sec:preliminaries}

The theory of regular coverings and of lifting automorphisms is usually developed in the dart model of graphs (see, for example,~\cite{MMP2004,MNS2000}), which has proved more convenient for this purpose. We therefore work in the dart model and follow~\cite[Section~2]{MMP2004}, with notation adapted to the present setting.

\subsection{Graphs}

A \emph{graph} is a quadruple
\[
X=(V(X),D(X),\Beg_X,\Inv_X),
\]
where $V(X)$ and $D(X)$ are disjoint finite sets, $\Beg_X\colon D(X)\to V(X)$ is a map assigning to each dart its \emph{initial vertex}, and $\Inv_X$ is an involutory permutation of $D(X)$. 

For $d\in D(X)$, we write $d^{-1}=\Inv_X d$ for the \emph{inverse} of $d$, and define $\End_X d=\Beg_X d^{-1}$ to be the \emph{terminal vertex} of $d$. For convenience, we write $d\colon u\to v$ to indicate that $d$ is a dart with initial vertex $u$ and terminal vertex $v$. The orbits of $\Inv_X$ on $D(X)$ are called the \emph{edges} of $X$. An edge $\{d,d^{-1}\}$ is called a \emph{semiedge} if $d=d^{-1}$, a \emph{loop} if $d\ne d^{-1}$ and $\Beg_X d=\End_X d$, and a \emph{link} otherwise. Distinct links with the same endvertices are called \emph{parallel}. When the graph under consideration is clear from the context, the subscript $X$ is omitted.

\subsection{Walks and homology}\label{subsec:fundamental}

A \emph{walk} in $X$ is a sequence $W=(d_1,\ldots,d_\ell)$ of darts satisfying $\End d_i=\Beg d_{i+1}$ for $1\leq i<\ell$. We also regard each vertex $v$ as the trivial walk of length $0$ from $v$ to itself. The \emph{initial vertex} and \emph{terminal vertex} of $W$ are defined by $\Beg W=\Beg d_1$ and $\End W=\End d_\ell$, respectively, and the \emph{inverse} of $W$ is
\[
W^{-1}=(d_\ell^{-1},\ldots,d_1^{-1}).
\]
A walk is called closed if $\Beg W=\End W$. A graph $X$ is said to be \emph{connected} if, for every pair of vertices $u,v\in V(X)$, there exists a walk $W$ in $X$ with $\Beg W=u$ and $\End W=v$.

For a prime $p$, we use the first homology group $H_1(X;\F_p)$ in the combinatorial sense of~\cite[p.~73]{MMP2004}. We recall only the description needed below. Let $X$ be connected and let $T$ be a spanning tree of $X$. The links and loops of $X$ not contained in $T$ are called \emph{cotree} links and loops, and their darts are called \emph{cotree darts}. For a cotree dart $d$, let $Q_d$ be the unique walk in $T$ from $\End d$ to $\Beg d$ with no repeated vertices. The closed walk $C_d=dQ_d$ is called the \emph{fundamental closed walk} determined by $d$; it traverses $d$ once and otherwise uses only darts of $T$, like the fundamental cycle of a cotree edge in ordinary graph theory. Choose one dart $d_i$ from each cotree link or loop, for $1\leq i\leq s$, and put $C_i=C_{d_i}$. If $p$ is odd, the homology classes of $C_1,\ldots,C_s$ form a basis of $H_1(X;\F_p)$. If $p=2$, the homology classes represented by the darts of the semiedges must be added to this basis.

\subsection{Morphisms}

A \emph{morphism} of graphs $f\colon X\to Y$ is a map
\[
f\colon V(X)\cup D(X)\longrightarrow V(Y)\cup D(Y)
\]
such that $f(V(X))\subseteq V(Y)$, $f(D(X))\subseteq D(Y)$, and for each $d\in D(X)$,
\[
f(\Beg_X d)=\Beg_Y(f(d))\ \text{ and }\ f(\Inv_X d)=\Inv_Y(f(d)).
\]

A graph morphism is an \emph{epimorphism} or an \emph{isomorphism} if it is surjective or bijective, respectively. An isomorphism from $X$ to itself is an \emph{automorphism}, and the group of all automorphisms of $X$ is denoted by $\Aut X$.

\medskip
\noindent\textbf{Convention.}
Throughout, isomorphisms act on the right, whereas all other maps are written in functional notation. In particular, if $f\colon X\to Y$ is a graph isomorphism and $x\in V(X)\cup D(X)$, then the image of $x$ under $f$ is denoted by $x^f$. If $g\colon Y\to Z$ is another graph isomorphism, then their composite from $X$ to $Z$ is denoted by $fg$, and $x^{fg}=(x^f)^g$.

\subsection{Quotient and covering projections}

Let $X$ be a graph and let $P\leq\Aut X$. The \emph{quotient graph} of $X$ relative to $P$, denoted by $X_P$, is the quadruple $(V(X_P),D(X_P),\Beg_{X_P},\Inv_{X_P})$, where
\[
V(X_P)=\{v^P\mid v\in V(X)\}\ \text{ and }\ D(X_P)=\{d^P\mid d\in D(X)\}
\]
are the sets of $P$-orbits on $V(X)$ and $D(X)$, respectively, and, for each $d\in D(X)$,
\[
\Beg_{X_P}d^P=(\Beg_X d)^P\ \text{ and }\ \Inv_{X_P}d^P=(\Inv_X d)^P.
\]
The natural epimorphism $\wp_P\colon X\to X_P$ sending each $x\in V(X)\cup D(X)$ to $x^P$ is called the \emph{quotient projection} relative to $P$.

For $v\in V(X)$, put $X(v)=\{d\in D(X)\mid \Beg_X d=v\}$. A graph epimorphism $\wp\colon\widetilde X\to X$ is called a \emph{covering projection} if, for every $\widetilde v\in V(\widetilde X)$, it maps $\widetilde X(\widetilde v)$ bijectively onto $X(\wp(\widetilde v))$. For $x\in V(X)\cup D(X)$, the set $\wp^{-1}(x)$ is called the \emph{fibre} over $x$.

Let $\wp_i\colon\widetilde X_i\to X_i$, $i=1,2$, be covering projections. An \emph{isomorphism of covering projections} $(\varphi,\widetilde\varphi)\colon \wp_1\to\wp_2$ is a pair of graph isomorphisms $\varphi\colon X_1\to X_2$ and $\widetilde\varphi\colon\widetilde X_1\to\widetilde X_2$ such that $\wp_2(x^{\widetilde\varphi})=\bigl(\wp_1(x)\bigr)^\varphi$ for every $x\in V(\widetilde X_1)\cup D(\widetilde X_1)$. Equivalently, the following diagram commutes.
\[
\begin{CD}
\widetilde X_1 @>{\widetilde\varphi}>> \widetilde X_2\\
@V{\wp_1}VV @VV{\wp_2}V\\
X_1 @>{\varphi}>> X_2
\end{CD}
\]
In this case, we say that $\varphi$ \emph{lifts along} $\wp_1$ and $\wp_2$ to $\widetilde\varphi$. If $X_1=X_2=X$, then $\wp_1$ and $\wp_2$ are \emph{equivalent} if there exists such an isomorphism of covering projections $(\id_X,\widetilde{\varphi}):\wp_1\to \wp_2$.

For a covering projection $\wp\colon\widetilde X\to X$, its \emph{group of covering transformations} is
\[
\CT(\wp)=\{\alpha\in\Aut\widetilde X\mid \wp(x^\alpha)=\wp(x)\text{ for every }x\in V(\widetilde X)\cup D(\widetilde X)\}.
\]
If $\widetilde X$ is connected, then $\CT(\wp)$ acts semiregularly on each fibre. The covering projection $\wp$ is called \emph{regular} if $\CT(\wp)$ acts transitively, and hence regularly, on each fibre.

\begin{proposition}[{\cite[Lemma~3]{PP2017}}]\label{prop:quotient-regular-cover}
If $X$ is a connected graph and $P\leq\Aut X$ is semiregular on $V(X)$, then $\wp_P$ is a regular covering projection and $\CT(\wp_P)=P$. 
\end{proposition}

\subsection{Ordinary voltages}\label{subsec:ordinary-voltages}

Let $X$ be a graph and let $N$ be a finite group. An \emph{$N$-voltage assignment} on $X$ is a map $\zeta\colon D(X)\to N$ such that $\zeta(d^{-1})=\zeta(d)^{-1}$ for every $d\in D(X)$. The element $\zeta(d)$ is called the \emph{voltage} of $d$. The \emph{derived graph} $X\times_\zeta N$ has vertex set $V(X)\times N$ and dart set $D(X)\times N$, with
\[
\Beg_{X\times_\zeta N}(d,n)=(\Beg_X d,n)
\ \text{ and }\
\Inv_{X\times_\zeta N}(d,n)=(d^{-1},n\zeta(d)).
\]
The projection
\[
\wp_\zeta\colon X\times_\zeta N\longrightarrow X,\ \ (x,n)\longmapsto x
\]
is a regular covering projection, called the \emph{derived covering projection}.

\begin{proposition}[{\cite[Section~2]{MMP2004}}]\label{prop:voltage-description}
Let $\wp\colon\widetilde X\to X$ be a connected regular covering projection, and let $N\cong\CT(\wp)$. Then there exist an $N$-voltage assignment $\zeta$ on $X$ and an equivalence of covering projections $(\id_X,\psi)\colon\wp_\zeta\to\wp$.
\end{proposition}

If $\wp_\zeta$ is equivalent to $\wp$, then $\zeta$ is called an \emph{$N$-voltage description} of $\wp$.

The voltage assignment extends to walks by
\[
\zeta(W)=\zeta(d_1)\cdots\zeta(d_\ell)
\]
for every walk $W=(d_1,\ldots,d_\ell)$. Suppose that $N$ is an elementary abelian $p$-group, written additively. Then the voltage of a closed walk depends only on its homology class, and $\zeta$ induces an $\F_p$-linear map 
\[
\overline\zeta\colon H_1(X;\F_p)\longrightarrow N,\ \ [W]\longmapsto\zeta(W).
\]

If $X$ is connected, then $X\times_\zeta N$ is connected if and only if every element of $N$ is the voltage of a closed walk based at some, equivalently every, vertex of $X$; see~\cite[Section~2]{MMP2004}. Consequently, when $N$ is elementary abelian, $\wp_\zeta$ is connected if and only if $\overline\zeta$ is surjective. Throughout this paper, all derived covering projections under consideration are connected.

\section{The quotient--lifting method}\label{sec:reduction}

In this section, we develop the quotient--lifting method described in the Introduction. Subsection~\ref{subsec:quotient-selection} concerns Step~\ref{step:intro-choose-quotients}, and Subsection~\ref{subsec:quotient-lifting} concerns Step~\ref{step:intro-lifting}. Step~\ref{step:intro-quotient-isomorphisms} depends on the structure of the quotient graphs and is carried out separately in each application.

\subsection{Choosing quotient projections}\label{subsec:quotient-selection}

After the subgroups $P$ and $P'$ in Step~\ref{step:intro-choose-quotients} have been chosen, the following lemma shows that an isomorphism $f$ between the original graphs satisfying $f^{-1}Pf=P'$ induces an isomorphism of the corresponding quotient projections. 

\begin{lemma}\label{lem:quotient-alignment}
Let $\Gamma$ and $\Gamma'$ be connected graphs, let $P\leq\Aut\Gamma$ and $P'\leq\Aut\Gamma'$ be subgroups acting semiregularly on vertices, and let $f\colon\Gamma\to\Gamma'$ be an isomorphism. If $f^{-1}Pf=P'$, then $f$ induces an isomorphism $\overline f\colon\Gamma_P\to\Gamma'_{P'}$ defined by $(x^P)^{\overline f}=(x^f)^{P'}$ for every $x\in V(\Gamma)\cup D(\Gamma)$. Moreover, $(\overline f,f)\colon\wp_P\to\wp_{P'}$ is an isomorphism of regular covering projections.
\end{lemma}

\begin{proof}
Let $x,y\in V(\Gamma)\cup D(\Gamma)$ and suppose that $x^P=y^P$. Then $y=x^\alpha$ for some $\alpha\in P$, and
\[
y^f=(x^\alpha)^f=(x^f)^{f^{-1}\alpha f},
\]
where $f^{-1}\alpha f\in P'$. Hence $(x^f)^{P'}=(y^f)^{P'}$, and so $\overline f$ is well defined. Since $f$ preserves initial vertices and dart inversion, so does $\overline f$. Moreover, since $fP'f^{-1}=P$, the analogous map induced by $f^{-1}$ is the inverse of $\overline f$. Hence $\overline f$ is an isomorphism from $\Gamma_P$ to $\Gamma'_{P'}$, and $\wp_{P'}(x^f)=\bigl(\wp_P(x)\bigr)^{\overline f}$ for every $x\in V(\Gamma)\cup D(\Gamma)$ by the definition of $\overline f$. By Proposition~$\ref{prop:quotient-regular-cover}$, both quotient projections are regular covering projections. Therefore $(\overline f,f)\colon\wp_P\to\wp_{P'}$ is an isomorphism of regular covering projections.
\end{proof}

To apply Lemma~$\ref{lem:quotient-alignment}$, it is enough to have an isomorphism $h\colon\Gamma\to\Gamma'$ for which $h^{-1}Ph$ is conjugate to $P'$ in $\Aut\Gamma'$. Indeed, if $\alpha\in\Aut\Gamma'$ satisfies $\alpha^{-1}(h^{-1}Ph)\alpha=P'$, then the isomorphism $f=h\alpha$ satisfies $f^{-1}Pf=P'$. The following proposition gives a setting, used in Section~\ref{sec:Q4p}, in which this conjugacy condition is automatic. To state it, we introduce the following terminology.

\begin{definition}\label{def:t-Cayley}
Let $G$ be a finite group and let $t$ be a positive integer. A \emph{$t$-Cayley graph} over $G$ is a graph $\Gamma$ equipped with an injective homomorphism
\[
\R_\Gamma\colon G\longrightarrow\Aut\Gamma
\]
such that $\R_\Gamma(G)$ acts semiregularly on $V(\Gamma)$ and has exactly $t$ orbits. Thus $\R_\Gamma(g)$ fixes a vertex only if $g=1$. For $H\leq G$, put
\[
\R_\Gamma(H)=\{\R_\Gamma(h)\mid h\in H\}.
\]
In particular, every $\R_\Gamma(G)$-orbit has cardinality $|G|$, and hence $|V(\Gamma)|=t|G|$.
\end{definition}

\begin{proposition}\label{prop:sylow-alignment}
Let $\Gamma$ be a connected $d$-valent $t$-Cayley graph over $G$, let $p>d$ be a prime with $p\nmid t$, and let $N\in\Syl_p(G)$. Then
\[
\R_\Gamma(N)\in\Syl_p(\Aut\Gamma).
\]
Consequently, if $\Gamma'$ is another connected $d$-valent $t$-Cayley graph over $G$ isomorphic to $\Gamma$, then there exists an isomorphism $f\colon\Gamma\to\Gamma'$ such that
\[
f^{-1}\R_\Gamma(N)f=\R_{\Gamma'}(N).
\]
\end{proposition}

\begin{proof}
We first prove that every $p$-subgroup of $\Aut\Gamma$ is semiregular on $V(\Gamma)$. Let $K$ be a $p$-subgroup of $\Aut\Gamma$. Suppose, to the contrary, that $K_v\neq1$ for some $v\in V(\Gamma)$. Since $K_v$ is a nontrivial $p$-group, it contains an element $q$ of order $p$. Every orbit of $\langle q\rangle$ on the set of the $d$ darts beginning at $v$ has cardinality either $1$ or $p$. Since $d<p$, every such orbit has cardinality $1$. Thus $q$ fixes every dart beginning at $v$, and hence fixes the terminal vertex of each such dart. The same argument shows that whenever $q$ fixes a vertex $u$, it fixes every dart beginning at $u$ and the terminal vertices of these darts. Since $\Gamma$ is connected, induction on the distance from $v$ shows that $q$ fixes every vertex and every dart of $\Gamma$. Hence $q=1$, a contradiction. Therefore $K_v=1$ for every $v\in V(\Gamma)$, and $K$ is semiregular on $V(\Gamma)$.

Every semiregular subgroup of $\Aut\Gamma$ has order dividing $|V(\Gamma)|=t|G|$. Hence the order of every $p$-subgroup of $\Aut\Gamma$ divides the $p$-part of $t|G|$. Since $p\nmid t$ and $N\in\Syl_p(G)$, this $p$-part is $|N|$. Moreover, the injectivity of $\R_\Gamma$ gives $|\R_\Gamma(N)|=|N|$. Therefore, $\R_\Gamma(N)\in\Syl_p(\Aut\Gamma)$. 

Now let $\Gamma'$ be as in the statement, and let $h\colon\Gamma\to\Gamma'$ be an isomorphism. By the first assertion of this proposition, $\R_\Gamma(N)\in\Syl_p(\Aut\Gamma)$ and $\R_{\Gamma'}(N)\in\Syl_p(\Aut\Gamma')$. Hence $h^{-1}\R_\Gamma(N)h$ and $\R_{\Gamma'}(N)$ are Sylow $p$-subgroups of $\Aut\Gamma'$.  They are therefore conjugate in $\Aut\Gamma'$. By the observation following Lemma~$\ref{lem:quotient-alignment}$, there exists an isomorphism $f\colon\Gamma\to\Gamma'$ satisfying $f^{-1}\R_\Gamma(N)f=\R_{\Gamma'}(N)$.
\end{proof}

\subsection{Lifting isomorphisms between quotient graphs}\label{subsec:quotient-lifting}

For the quotient projections $\wp_P$ and $\wp_{P'}$ chosen as in Subsection~\ref{subsec:quotient-selection}, we have $\Gamma\cong\Gamma'$ if and only if there exists an isomorphism $\varphi\colon\Gamma_P\to\Gamma'_{P'}$ that lifts along $\wp_P$ and $\wp_{P'}$ to an isomorphism from $\Gamma$ to $\Gamma'$. Thus Step~\ref{step:intro-lifting} reduces to determining which isomorphisms between the quotient graphs lift.

The lifting criterion of~\cite[Corollary~3.3]{MMP2004} is formulated for two coverings of the same base graph, whereas here the base graphs $\Gamma_P$ and $\Gamma'_{P'}$ are merely isomorphic. To reduce to the common-base setting, we use the standard pullback construction.

Let $\varphi\colon X\to Y$ be an isomorphism, and let $\eta\colon D(Y)\to M$ be a voltage assignment. The \emph{pullback} of $\eta$ along $\varphi$ is the voltage assignment
\[
\varphi^*\eta\colon D(X)\longrightarrow M,\ \ (\varphi^*\eta)(d)=\eta(d^\varphi).
\]
Define
\[
\widehat\varphi\colon X\times_{\varphi^*\eta}M\longrightarrow Y\times_\eta M,\ \ (x,a)^{\widehat\varphi}=(x^\varphi,a)
\]
for $x\in V(X)\cup D(X)$ and $a\in M$. Then $\widehat\varphi$ is an isomorphism, and
\[
\wp_\eta\bigl((x,a)^{\widehat\varphi}\bigr)=x^\varphi=\bigl(\wp_{\varphi^*\eta}(x,a)\bigr)^\varphi
\]
for every $x\in V(X)\cup D(X)$ and $a\in M$. Therefore
\[
(\varphi,\widehat\varphi)\colon\wp_{\varphi^*\eta}\longrightarrow\wp_\eta
\]
is an isomorphism of covering projections.

\begin{proposition}\label{prop:cover-comparison}
Let
\[
\wp_\zeta\colon X\times_\zeta N\longrightarrow X\ \text{ and }\ \wp_\eta\colon Y\times_\eta M\longrightarrow Y
\]
be connected derived covering projections, and let $\varphi\colon X\to Y$ be an isomorphism. Then there exists an isomorphism of covering projections $(\varphi,\widetilde\varphi)\colon\wp_\zeta\longrightarrow\wp_\eta$ if and only if, for some $v\in V(X)$, there is a group isomorphism $\alpha\colon N\to M$ such that
\begin{equation}\label{eq:local-voltage-comparison}
\bigl(\zeta(W)\bigr)^\alpha=\eta(W^\varphi)
\end{equation}
for every closed walk $W$ in $X$ based at $v$.
\end{proposition}

\begin{proof}
Consider the following diagram.
\begingroup
\setlength{\abovedisplayskip}{0.5\baselineskip}
\setlength{\belowdisplayskip}{0.7\baselineskip}
\[
\begin{CD}
X\times_\zeta N
@>{\widetilde\psi}>>
X\times_{\varphi^*\eta}M
@>{\widehat\varphi}>>
Y\times_\eta M\\
@V{\wp_\zeta}VV
@V{\wp_{\varphi^*\eta}}VV
@V{\wp_\eta}VV\\
X
@>{\id_X}>>
X
@>{\varphi}>>
Y.
\end{CD}
\]
\endgroup
Since $\widehat\varphi$ is an isomorphism and the right-hand square commutes, an isomorphism of covering projections $(\varphi,\widetilde\varphi)\colon\wp_\zeta\to\wp_\eta$ exists if and only if there exists an isomorphism $\widetilde\psi$ making the left-hand square commute. By definition, the latter condition holds if and only if $\wp_\zeta$ and $\wp_{\varphi^*\eta}$ are equivalent covering projections over $X$. By \cite[Corollary~3.3(b)]{MMP2004}, this is equivalent to the existence of a vertex $v\in V(X)$ and a group isomorphism $\alpha\colon N\to M$ such that, for every closed walk $W$ in $X$ based at $v$,
\[
\bigl(\zeta(W)\bigr)^\alpha=(\varphi^*\eta)(W)=\eta(W^\varphi).\qedhere
\]
\end{proof}

Proposition~\ref{prop:cover-comparison} characterizes liftability in terms of the voltages of closed walks. When $N$ and $M$ are elementary abelian $p$-groups, the voltages of closed walks depend only on their classes in first homology over $\F_p$, as described in Subsection~\ref{subsec:ordinary-voltages}. For an isomorphism $\varphi\colon X\to Y$, let
\[
\varphi_*\colon H_1(X;\F_p)\longrightarrow H_1(Y;\F_p)
\]
be the induced $\F_p$-linear isomorphism, defined by $[W]^{\varphi_*}=[W^\varphi]$ for every closed walk $W$ in $X$. The following corollary reformulates the lifting condition in terms of linear maps on first homology.

\begin{corollary}\label{cor:homology-comparison}
In the setting of Proposition~$\ref{prop:cover-comparison}$, suppose further that $N$ and $M$ are elementary abelian $p$-groups. Let
\[
\overline\zeta\colon H_1(X;\F_p)\longrightarrow N\ \text{ and }\ \overline\eta\colon H_1(Y;\F_p)\longrightarrow M
\]
be the $\F_p$-linear maps induced by the voltage assignments $\zeta$ and $\eta$, respectively, as defined in Subsection~$\ref{subsec:ordinary-voltages}$. Then the isomorphism $\varphi$ lifts along $\wp_\zeta$ and $\wp_\eta$ if and only if there exists an $\F_p$-linear isomorphism $\alpha\colon N\to M$ such that
\begin{equation}\label{eq:homology-comparison}
\bigl(\overline\zeta(c)\bigr)^\alpha=\overline\eta\bigl(c^{\varphi_*}\bigr)
\end{equation}
for every $c\in H_1(X;\F_p)$.
Moreover, if such an $\alpha$ exists, then it is unique.
\end{corollary}

\begin{proof}
Every group isomorphism from $N$ to $M$ is automatically $\F_p$-linear. If $W$ is a closed walk in $X$, then
\[
\zeta(W)=\overline\zeta([W])\ \text{ and }\ \eta(W^\varphi)=\overline\eta([W^\varphi])=\overline\eta\bigl([W]^{\varphi_*}\bigr).
\]
Since $X$ is connected, every class in $H_1(X;\F_p)$ can be represented by a closed walk based at any prescribed vertex of $X$. Therefore, the condition in Proposition~$\ref{prop:cover-comparison}$ is equivalent to \eqref{eq:homology-comparison}.

Since $\wp_\zeta$ is connected, $\overline\zeta$ is surjective by Subsection~$\ref{subsec:ordinary-voltages}$. Suppose that $\alpha$ and $\beta$ both satisfy \eqref{eq:homology-comparison}, and let $n\in N$. Choose $c\in H_1(X;\F_p)$ such that $\overline\zeta(c)=n$. Then
\[
n^\alpha=\bigl(\overline\zeta(c)\bigr)^\alpha
=\overline\eta\bigl(c^{\varphi_*}\bigr)
=\bigl(\overline\zeta(c)\bigr)^\beta
=n^\beta.
\]
Since $n$ was arbitrary, $\alpha=\beta$.
\end{proof}

We now return to the setting of $t$-Cayley graphs. Combining the quotient-selection results of Subsection~\ref{subsec:quotient-selection} with Corollary~\ref{cor:homology-comparison} gives the following quotient--lifting test for graph isomorphism.

\begin{corollary}\label{cor:quotient-lifting-test}
Let $\Gamma$ and $\Gamma'$ be connected $d$-valent $t$-Cayley graphs over a group $G$. Let $p>d$ be a prime with $p\nmid t$, and suppose that $N\in\Syl_p(G)$ satisfies $N\cong C_p^r$ for some $r\geq 1$. Put
\[
P=\R_\Gamma(N),\ \ P'=\R_{\Gamma'}(N),\ \ X=\Gamma_P,\ \ Y=\Gamma'_{P'}.
\]
Choose $N$-voltage descriptions $\zeta\colon D(X)\to N$ and $\eta\colon D(Y)\to N$ of $\wp_P$ and $\wp_{P'}$, respectively, and let $\overline\zeta$ and $\overline\eta$ be the induced maps on first homology. Then $\Gamma\cong\Gamma'$ if and only if there exist an isomorphism $\varphi\colon X\to Y$ and an $\F_p$-linear automorphism $\alpha$ of $N$ such that
\[
\bigl(\overline\zeta(c)\bigr)^\alpha=\overline\eta\bigl(c^{\varphi_*}\bigr)
\]
for every $c\in H_1(X;\F_p)$.
\end{corollary}

\begin{proof}
Let $(\id_X,\psi)\colon\wp_\zeta\to\wp_P$ and $(\id_Y,\psi')\colon\wp_\eta\to\wp_{P'}$ be the equivalences of covering projections given by the voltage descriptions $\zeta$ and $\eta$.

Suppose first that $\Gamma\cong\Gamma'$. By Proposition~$\ref{prop:sylow-alignment}$, there exists an isomorphism $f\colon\Gamma\to\Gamma'$ such that $f^{-1}Pf=P'$. By Lemma~$\ref{lem:quotient-alignment}$, $f$ induces an isomorphism $\varphi\colon X\to Y$ such that $(\varphi,f)\colon\wp_P\to\wp_{P'}$ is an isomorphism of covering projections. For every vertex or dart $x$ of $X\times_\zeta N$, we have
\[
\wp_\eta\bigl(x^{\psi f\psi'^{-1}}\bigr)=\wp_{P'}\bigl(x^{\psi f}\bigr)=\bigl(\wp_P(x^\psi)\bigr)^{\varphi}=\bigl(\wp_\zeta(x)\bigr)^{\varphi}.
\]
Hence $(\varphi,\psi f\psi'^{-1})\colon\wp_\zeta\to\wp_\eta$ is an isomorphism of covering projections, that is, $\varphi$ lifts along $\wp_\zeta$ and $\wp_\eta$. Corollary~$\ref{cor:homology-comparison}$ therefore gives an $\F_p$-linear automorphism $\alpha$ of $N$ satisfying the required equality.

Conversely, suppose that $\varphi$ and $\alpha$ satisfy the stated equality. By Corollary~$\ref{cor:homology-comparison}$, $\varphi$ lifts along $\wp_\zeta$ and $\wp_\eta$ to an isomorphism $\widetilde\varphi\colon X\times_\zeta N\to Y\times_\eta N$. Then $\psi^{-1}\widetilde\varphi\psi'$ is an isomorphism from $\Gamma$ to $\Gamma'$.
\end{proof}

\begin{remark}\label{rem:computational-lifting}
For fixed quotient graphs and fixed voltage assignments, Corollary~$\ref{cor:quotient-lifting-test}$ gives a finite test. In the setting of that corollary, let $(c_1,\ldots,c_\beta)$ be the basis of $H_1(X;\F_p)$ obtained from a spanning tree of $X$ as in Subsection~\ref{subsec:fundamental}, with $c_i=[C_i]$. Identify $N$ with $\F_p^r$ by choosing an ordered basis of $N$, and regard the elements of $N$ as column vectors; then every $\F_p$-linear automorphism $\alpha$ of $N$ is of the form $u\mapsto Au$ for a unique $A\in\GL(r,p)$. Put
\[
Z=
\begin{bmatrix}
\zeta(C_1)&\cdots&\zeta(C_\beta)
\end{bmatrix}
\in\F_p^{r\times\beta},
\]
and, for each isomorphism $\varphi\colon X\to Y$, put
\[
Z_\varphi'=
\begin{bmatrix}
\eta(C_1^\varphi)&\cdots&\eta(C_\beta^\varphi)
\end{bmatrix}
\in\F_p^{r\times\beta}.
\]
Since $c_1,\ldots,c_\beta$ form a basis of $H_1(X;\F_p)$, the equality in Corollary~$\ref{cor:quotient-lifting-test}$ holds for all $c$ if and only if it holds for $c_1,\ldots,c_\beta$, that is, if and only if $Z_\varphi'=AZ$. Hence, by Corollary~$\ref{cor:homology-comparison}$, $\varphi$ lifts along $\wp_\zeta$ and $\wp_\eta$ if and only if $Z_\varphi'=AZ$ for some $A\in\GL(r,p)$. Since $\overline\zeta$ is surjective, $\rank Z=r$. Choose $1\leq j_1<\cdots<j_r\leq\beta$ such that the submatrix $Z_0$ of $Z$ formed by the columns $j_1,\ldots,j_r$ is nonsingular, and let $Z'_{\varphi,0}$ be the submatrix of $Z'_\varphi$ formed by the same columns. Comparing these columns in $Z_\varphi'=AZ$ shows that such an $A$ must equal $Z'_{\varphi,0}Z_0^{-1}$. Consequently, $\varphi$ lifts along $\wp_\zeta$ and $\wp_\eta$ if and only if
\[
A_\varphi:=Z'_{\varphi,0}Z_0^{-1}\in\GL(r,p)
\quad\text{and}\quad
Z'_\varphi=A_\varphi Z.
\]
The isomorphisms $\varphi\colon X\to Y$ must be enumerated in the dart model, recording the image of every dart, because parallel links are distinct and $\eta(C_i^\varphi)$ depends on the dart sequence $C_i^\varphi$. In Section~\ref{sec:Q4p}, the quotient graphs and voltage assignments depend on parameters, and we carry out this test for all values of the parameters at once by means of integer matrices.
\end{remark}

\section{Tetravalent bi-Cayley graphs over \texorpdfstring{$Q_{4p}$}{Q4p}}
\label{sec:Q4p}

Throughout this section we use the notation $\BCay(G,S)$ for bi-Cayley graphs introduced in the Introduction, and we take $\R_\Gamma$ to be the right regular representation defined by
\[
(x,i)^{\R_\Gamma(g)}=(xg,i)\ \text{ for all }\ g,x\in G\ \text{ and }\ i\in\{0,1\}.
\]
Then $\R_\Gamma(G)$ acts semiregularly on $V(\Gamma)$ and has exactly the two orbits $G\times\{0\}$ and $G\times\{1\}$. Thus $\BCay(G,S)$ is an $|S|$-valent $2$-Cayley graph over $G$ in the sense of Definition~\ref{def:t-Cayley}.

We shall use the following description of the components of a bi-Cayley graph, whose elementary proof is omitted.

\begin{lemma}\label{lem:bcay-components}
Let $G$ be a finite group and let $S\subseteq G$ with $1\in S$. Then the connected components of $\BCay(G,S)$ are the subgraphs induced by the sets $\langle S\rangle x\times\{0,1\}$ with $x\in G$. Each of them has $2|\langle S\rangle|$ vertices and is isomorphic to $\BCay(\langle S\rangle,S)$. In particular, $\BCay(G,S)$ is connected if and only if $\langle S\rangle=G$.
\end{lemma}

From now on, let
\[
Q=Q_{4p}=\langle a,b\mid a^{2p}=1,\ b^2=a^p,\ b^{-1}ab=a^{-1}\rangle,
\]
where $p\geq3$ is prime, and regard the exponents of $a$ as elements of $\Z_{2p}$. It is well known (see, for example,~\cite{XieZhang2026}) that $\Aut Q=\{\sigma_{r,t}\mid r\in\Z_{2p}^*,\ t\in\Z_{2p}\}$, where
\begin{equation}\label{eq:Q4p-automorphisms}
    (a^i)^{\sigma_{r,t}}=a^{ri}\ \text{ and }\ (ba^i)^{\sigma_{r,t}}=ba^{ri+t}\ \text{ for all }\ i\in\Z_{2p}.
\end{equation}

The main task in proving Theorem~\ref{thm:intro-Q4p-4BCI} is to determine the isomorphisms between bi-Cayley graphs whose connection sets have four elements, contain $1$, and are not contained in $\langle a\rangle$. We will see in Lemma~\ref{lem:Q4p-four-set-normal-form} that each such connection set is BCI-equivalent to a member of one of the two families introduced below. We first solve the isomorphism problem for these two families.

Define two parameter sets
\[
\mathcal P_C=\{(k,\ell)\in(\Z_{2p}\setminus\{0\})^2\mid k\ne\ell\}\ \text{ and }\ \mathcal P_M=(\Z_{2p}\setminus\{0\})^2.
\]
The \emph{$C$-family} and the \emph{$M$-family} consist of the connection sets
\[
C(k,\ell)=\{1,a^k,a^\ell,b\}\ \text{ with }\ (k,\ell)\in\mathcal P_C
\quad\text{and}\quad
M(k,\ell)=\{1,a^k,ba^\ell,b\}\ \text{ with }\ (k,\ell)\in\mathcal P_M,
\]
respectively, and we write $\Gamma_C(k,\ell)=\BCay(Q,C(k,\ell))$ and $\Gamma_M(k,\ell)=\BCay(Q,M(k,\ell))$. All these connection sets have cardinality $4$, so all these graphs are tetravalent.

Let $U=\Z_{2p}^*$ and let $S_3$ denote the symmetric group on $\{0,1,2\}$. For $(k,\ell)\in\mathcal P_C$, put
\begin{equation}\label{eq:e-and-x}
(e_0,e_1,e_2)=(0,k,\ell).
\end{equation}
To describe the parameter pairs related by the BCI-equivalences established in Lemma~\ref{lem:Q4p-evident-transformations} below, define
\begin{align*}
\mathcal O_C(k,\ell)&=\bigl\{r\bigl(e_{\rho(1)}-e_{\rho(0)},\ e_{\rho(2)}-e_{\rho(0)}\bigr)\bigm| r\in U,\ \rho\in S_3\bigr\},\\
\mathcal O_M(k,\ell)&=\bigl\{r(k,\ell),\ r(k,-\ell),\ r(\ell,k),\ r(\ell,-k)\bigm| r\in U\bigr\}.
\end{align*}
All coordinates in these sets are calculated modulo $2p$.

\begin{lemma}\label{lem:Q4p-evident-transformations}
The following statements hold.
\begin{enumerate}[label=\textup{(\roman*)}]
\item Let $(k,\ell),(k',\ell')\in\mathcal P_C$. If $(k',\ell')\in\mathcal O_C(k,\ell)$, then $C(k,\ell)$ and $C(k',\ell')$ are BCI-equivalent.
\item Let $(k,\ell),(k',\ell')\in\mathcal P_M$. If $(k',\ell')\in\mathcal O_M(k,\ell)$, then $M(k,\ell)$ and $M(k',\ell')$ are BCI-equivalent.
\end{enumerate}
\end{lemma}

\begin{proof}
BCI-equivalence is an equivalence relation, and by~\eqref{eq:Q4p-automorphisms} the automorphism $\sigma_{r,0}$ multiplies both parameters of $C(k,\ell)$ and of $M(k,\ell)$ by $r\in U$. It therefore suffices to realise the unscaled pairs.

For~(i), let $\rho\in S_3$ and put $m=e_{\rho(0)}$. Using $a^{-m}b=ba^{m}$ we obtain
\[
a^{-m}\,C(k,\ell)^{\sigma_{1,-m}}=\bigl\{a^{e_0-m},a^{e_1-m},a^{e_2-m},b\bigr\}=C\bigl(e_{\rho(1)}-e_{\rho(0)},\ e_{\rho(2)}-e_{\rho(0)}\bigr),
\]
the last equality follows from $e_{\rho(0)}-m=0$ and $C(x,y)=C(y,x)$. As $e_0,e_1,e_2$ are pairwise distinct, the resulting pair lies in $\mathcal P_C$.

For~(ii), the four unscaled pairs are realised by
\[
M(k,\ell),\ \  M(k,\ell)^{\sigma_{1,-\ell}}=M(k,-\ell),\ \  b\,M(k,\ell)^{\sigma_{1,p}}=M(\ell,k),\ \ \bigl(b\,M(k,\ell)^{\sigma_{1,p}}\bigr)^{\sigma_{1,-k}}=M(\ell,-k).\qedhere
\]
\end{proof}

Thus membership in $\mathcal O_C(k,\ell)$ or $\mathcal O_M(k,\ell)$ is sufficient for the corresponding graphs to be isomorphic. The following proposition, the main technical result of this section, shows that it is also necessary, and that no graph of the $C$-family is isomorphic to a graph of the $M$-family.

\begin{proposition}\label{prop:Q4p-classification}
The following statements hold.
\begin{enumerate}[label=\textup{(\roman*)},ref=\textup{(\roman*)}]
\item\label{prop:Q4p-classification-C}
For $(k,\ell),(k',\ell')\in\mathcal P_C$, we have $\Gamma_C(k,\ell)\cong\Gamma_C(k',\ell')$ if and only if $(k',\ell')\in\mathcal O_C(k,\ell)$.

\item\label{prop:Q4p-classification-M}
For $(k,\ell),(k',\ell')\in\mathcal P_M$, we have $\Gamma_M(k,\ell)\cong\Gamma_M(k',\ell')$ if and only if $(k',\ell')\in\mathcal O_M(k,\ell)$.

\item\label{prop:Q4p-classification-CM}
For $(k,\ell)\in\mathcal P_C$ and $(k',\ell')\in\mathcal P_M$, we have $\Gamma_C(k,\ell)\not\cong\Gamma_M(k',\ell')$.
\end{enumerate}
Moreover, whenever two graphs in~\ref{prop:Q4p-classification-C} or~\ref{prop:Q4p-classification-M} are isomorphic, their connection sets are BCI-equivalent.
\end{proposition}

For $p\geq5$, we prove Proposition~\ref{prop:Q4p-classification} by Steps~\ref{step:intro-choose-quotients}--\ref{step:intro-lifting} of the Introduction. Step~\ref{step:intro-choose-quotients} is given by Proposition~\ref{prop:sylow-alignment} (see the paragraph after Notation~\ref{not:Q4p-quotient}). For Step~\ref{step:intro-quotient-isomorphisms}, Lemma~\ref{lem:Q4p-quotient-voltage} and Definition~\ref{def:quotient-type} identify the quotient graphs, Lemma~\ref{lem:Q4p-quotient-separation} shows that quotient graphs of the two families are not isomorphic, and the isomorphisms between quotient graphs of the same family are enumerated by computer (Appendix~\ref{apdx:pge5}). Step~\ref{step:intro-lifting} occupies Lemmas~\ref{lem:Q4p-common-tree}--\ref{lem:Q4p-recovery}. The case $p=3$ is settled in Appendix~\ref{apdx:p3}.

The following lemma allows us to restrict the remaining arguments to connected graphs.

\begin{lemma}\label{lem:Q4p-connectivity}
For every $(k,\ell)\in\mathcal P_C$, the graph $\Gamma_C(k,\ell)$ is connected. For $(k,\ell)\in\mathcal P_M$, the graph $\Gamma_M(k,\ell)$ is disconnected if and only if $k=\ell=p$.
\end{lemma}

\begin{proof}
By Lemma~\ref{lem:bcay-components} it suffices to decide when $\langle S\rangle=Q$.

Assume first that $S=C(k,\ell)$. Then $\langle S\rangle$ contains $b$, hence also $b^2=a^p$, as well as $a^k$ and $a^\ell$. Thus $\langle a^p,a^k,a^\ell\rangle\leq\langle S\rangle$, and $\langle a^p,a^k,a^\ell\rangle=\langle a^c\rangle$ with $c=\gcd(p,k,\ell)\in\{1,p\}$. Here $c=p$ would force $k,\ell\in\{0,p\}$ and hence $k=\ell=p$, which is impossible for $(k,\ell)\in\mathcal P_C$; so $c=1$, $\langle a\rangle\leq\langle S\rangle$ and therefore $\langle S\rangle=Q$.

Assume now that $S=M(k,\ell)$. Then $\langle S\rangle$ contains $b$, $a^k$ and $a^\ell=b^{-1}(ba^\ell)$, so the same argument gives $\langle S\rangle=Q$ unless $k=\ell=p$. If $k=\ell=p$, then $M(p,p)=\{1,a^p,ba^p,b\}=\langle b\rangle$ has order $4$, and $\Gamma_M(p,p)$ is disconnected.
\end{proof}

We next introduce the notation for the quotient projections modulo the Sylow $p$-subgroup of $Q$ and their voltage assignments.

\begin{notation}\label{not:Q4p-quotient}
Suppose that $p\geq5$, and let $\Gamma=\Gamma_C(k,\ell)$ or $\Gamma_M(k,\ell)$ be connected. Put
\[
z=a^2,\ \ N=\langle z\rangle,\ \ P_\Gamma=\R_\Gamma(N).
\]
Then $N$ is the unique Sylow $p$-subgroup of $Q$ and has order $p$. We identify $N$ with the additive group of $\F_p$ by $z^t\mapsto t$. For $m\in\Z_{2p}$, let $\epsilon_m\in\{0,1\}$ denote the parity of $m$, and let $\widehat m\in\F_p$ be the element with $2\widehat m=\overline m$, where $\overline m\in\F_p$ is the image of $m$ under the natural homomorphism $\Z_{2p}\to\F_p$. Then
\[
a^m=b^{2\epsilon_m}z^{\widehat m};
\]
indeed, if $m=2j$ is even, then $a^m=z^j$ and $2j=m$, whereas if $m$ is odd, then $m-p=2j$ is even, $a^m=a^pa^{2j}=b^2z^j$ and $2j\equiv m\pmod p$. Every element of $Q$ has a unique expression $b^iz^t$ with $i\in\Z_4$ and $t\in\F_p$. For $i\in\Z_4$, let $u_i=(b^i,0)^{P_\Gamma}$ and $v_i=(b^i,1)^{P_\Gamma}$; these are vertices of the quotient graph $\Gamma_{P_\Gamma}$. List the connection set of $\Gamma$ as
\[
(s_1,s_2,s_3,s_4)=
\begin{cases}
(1,a^k,a^\ell,b),&\text{if }\Gamma=\Gamma_C(k,\ell),\\
(1,a^k,ba^\ell,b),&\text{if }\Gamma=\Gamma_M(k,\ell),
\end{cases}
\]
and, for $i\in\Z_4$ and $1\leq j\leq4$, let $d_i(j)$ denote the $P_\Gamma$-orbit of the dart $(b^i,0)\to(s_jb^i,1)$. Finally, put $\kappa=\widehat k$ and $\lambda=\widehat\ell$. When a second connected graph $\Gamma'=\Gamma_C(k',\ell')$ or $\Gamma_M(k',\ell')$ is considered simultaneously, we use the corresponding primed notation; in particular, $\kappa'=\widehat{k'}$ and $\lambda'=\widehat{\ell'}$.
\end{notation}

Under Notation~\ref{not:Q4p-quotient}, $\Gamma$ is a connected $4$-valent $2$-Cayley graph over $Q$, and $N\cong C_p$ is a Sylow $p$-subgroup of $Q$ with $p>4$ and $p\nmid2$. Hence Proposition~\ref{prop:sylow-alignment} and Corollary~\ref{cor:quotient-lifting-test} apply to any two such graphs. The hypothesis $p\geq5$ is used here and in Lemma~\ref{lem:Q4p-finite-computation}, where $p$ must not divide the Smith invariants $2$ and $3$.

\begin{lemma}\label{lem:Q4p-quotient-voltage}
Under Notation~$\ref{not:Q4p-quotient}$, put $X=\Gamma_{P_\Gamma}$. The quotient projection $\wp_{P_\Gamma}\colon\Gamma\to X$ admits an $\F_p$-voltage description $\zeta_\Gamma$. Its values on the darts $d_i(j)$ are listed in Table~$\ref{tab:Q4p-quotient-voltages}$, where all subscripts are in $\Z_4$, and $\zeta_\Gamma(d_i(j)^{-1})=-\zeta_\Gamma(d_i(j))$.
\end{lemma}

\begin{table}[!htb]
\normalfont
\centering
\begin{minipage}[t]{0.48\textwidth}
\centering
$\Gamma=\Gamma_C(k,\ell)$\par
\vspace{0.5em}
\begin{tabular}{@{}c|l|c@{}}
$s_j$&\multicolumn{1}{c|}{dart of $X$}&voltage $\zeta_\Gamma$\\ \hline
$1$&$d_i(1)\colon u_i\to v_i$&$0$\\
$a^k$&$d_i(2)\colon u_i\to v_{i+2\epsilon_k}$&$(-1)^i\kappa$\\
$a^\ell$&$d_i(3)\colon u_i\to v_{i+2\epsilon_\ell}$&$(-1)^i\lambda$\\
$b$&$d_i(4)\colon u_i\to v_{i+1}$&$0$
\end{tabular}
\end{minipage}
\hfill
\begin{minipage}[t]{0.48\textwidth}
\centering
$\Gamma=\Gamma_M(k,\ell)$\par
\vspace{0.5em}
\begin{tabular}{@{}c|l|c@{}}
$s_j$&\multicolumn{1}{c|}{dart of $X$}&voltage $\zeta_\Gamma$\\ \hline
$1$&$d_i(1)\colon u_i\to v_i$&$0$\\
$a^k$&$d_i(2)\colon u_i\to v_{i+2\epsilon_k}$&$(-1)^i\kappa$\\
$ba^\ell$&$d_i(3)\colon u_i\to v_{i+1+2\epsilon_\ell}$&$(-1)^i\lambda$\\
$b$&$d_i(4)\colon u_i\to v_{i+1}$&$0$
\end{tabular}
\end{minipage}
\par\vspace{0.8em}
\caption{The quotient darts and their $\zeta_\Gamma$-voltages.}
\label{tab:Q4p-quotient-voltages}
\end{table}

\begin{proof}
Since $z^tb^i=b^iz^{(-1)^it}$ and $a^m=b^{2\epsilon_m}z^{\widehat m}$, for all $m\in\Z_{2p}$, $i\in\Z_4$ and $t\in\F_p$ we obtain
\[
a^mb^iz^t=b^{i+2\epsilon_m}z^{\,t+(-1)^i\widehat m}\ \text{ and }\ ba^mb^iz^t=b^{i+1+2\epsilon_m}z^{\,t+(-1)^i\widehat m}.
\]
Hence, if a row of Table~\ref{tab:Q4p-quotient-voltages} lists the dart $d_i(j)\colon u_i\to v_{i'}$ with voltage $w$, then
\begin{equation}\label{eq:Q4p-dart-ends}
s_jb^iz^t=b^{i'}z^{\,t+w}\ \text{ for every }\ t\in\F_p;
\end{equation}
for $t=0$, this gives the incidence $d_i(j)\colon u_i\to v_{i'}$.

Since $\R_\Gamma(z^t)$ maps $(x,e)$ to $(xz^t,e)$, the orbit $u_i$ consists of the vertices $(b^iz^t,0)$, the orbit $v_i$ of the vertices $(b^iz^t,1)$, and the orbit $d_i(j)$ of the darts $(b^iz^t,0)\to(s_jb^iz^t,1)$, where $t$ runs through $\F_p$. As every element of $Q$ is uniquely of the form $b^iz^t$, the eight orbits $u_i,v_i$ are pairwise distinct and form $V(X)$, and the sixteen orbits $d_i(j)$ are pairwise distinct and, together with their inverses, form $D(X)$. Define $\psi\colon X\times_{\zeta_\Gamma}\F_p\to\Gamma$ by
\[
(u_i,t)^\psi=(b^iz^t,0),\quad (v_i,t)^\psi=(b^iz^t,1),\quad (d_i(j),t)^\psi=\bigl((b^iz^t,0)\to(s_jb^iz^t,1)\bigr),
\]
and let $\psi$ map the inverse $(d_i(j)^{-1},t+w)$ of $(d_i(j),t)$ to the inverse of $(d_i(j),t)^\psi$. By the description of the orbits, $\psi$ is bijective on vertices and on darts, and $(x,t)^\psi$ lies in the orbit $x$ for every vertex or dart $x$ of $X$. By construction, $\psi$ commutes with dart inversion and maps the initial vertex $(u_i,t)$ of $(d_i(j),t)$ to the initial vertex of $(d_i(j),t)^\psi$. By~\eqref{eq:Q4p-dart-ends}, it maps the initial vertex $(v_{i'},t+w)$ of $(d_i(j)^{-1},t+w)$ to $(s_jb^iz^t,1)$, which is the initial vertex of the inverse of $(d_i(j),t)^\psi$. Hence $\psi$ is an isomorphism, and $(\id_X,\psi)\colon\wp_{\zeta_\Gamma}\to\wp_{P_\Gamma}$ is an equivalence of covering projections. Therefore $\zeta_\Gamma$ is an $\F_p$-voltage description of $\wp_{P_\Gamma}$.
\end{proof}
 
We now carry out the test of Remark~\ref{rem:computational-lifting} for all graphs in Notation~\ref{not:Q4p-quotient} at once. We first list the quotient graphs and voltage assignments that occur. By Table~\ref{tab:Q4p-quotient-voltages}, the quotient graph depends on $F\in\{C,M\}$ and on $k,\ell$ only through the pair $(\epsilon_k,\epsilon_\ell)$.

\begin{definition}\label{def:quotient-type}
A \emph{quotient type} is a triple $\tau=(F,\epsilon,\delta)\in\{C,M\}\times\{0,1\}^2$. Let $X_\tau$, also written $X_F^{\epsilon,\delta}$, be the graph with
\[
V(X_\tau)=\{u_i,v_i\mid i\in\Z_4\},\qquad D(X_\tau)=\{d_i(j),\,d_i(j)^{-1}\mid i\in\Z_4,\ 1\leq j\leq4\},
\]
where $\Inv_{X_\tau}$ interchanges $d_i(j)$ and $d_i(j)^{-1}$, and where, for $i\in\Z_4$,
\[
\begin{aligned}
F=C:\quad & d_i(1)\colon u_i\to v_i, && d_i(2)\colon u_i\to v_{i+2\epsilon}, && d_i(3)\colon u_i\to v_{i+2\delta}, && d_i(4)\colon u_i\to v_{i+1},\\
F=M:\quad & d_i(1)\colon u_i\to v_i, && d_i(2)\colon u_i\to v_{i+2\epsilon}, && d_i(3)\colon u_i\to v_{i+1+2\delta}, && d_i(4)\colon u_i\to v_{i+1},
\end{aligned}
\]
all subscripts being read in $\Z_4$. For $\Gamma=\Gamma_F(k,\ell)$ as in Notation~\ref{not:Q4p-quotient}, we call $\tau(\Gamma):=(F,\epsilon_k,\epsilon_\ell)$ the \emph{quotient type of $\Gamma$}; by Notation~\ref{not:Q4p-quotient} and Lemma~\ref{lem:Q4p-quotient-voltage}, $\Gamma_{P_\Gamma}=X_{\tau(\Gamma)}$.
\end{definition}

By Definition~\ref{def:quotient-type}, every dart $d_i(j)$ of $X_\tau$ joins a vertex $u_i$ to a vertex $v_{i'}$. Hence $X_\tau$ has neither loops nor semiedges, and
\begin{equation}\label{eq:dart-decomposition}
D(X_\tau)=\bigl\{d_i(j)\bigm| i\in\Z_4,\ 1\leq j\leq4\bigr\}\ \sqcup\ \bigl\{d_i(j)^{-1}\bigm| i\in\Z_4,\ 1\leq j\leq4\bigr\},
\end{equation}
so that $X_\tau$ has exactly sixteen links, each containing exactly one of the sixteen darts listed in Table~\ref{tab:Q4p-quotient-voltages}. Moreover, the links containing the darts $d_i(1)$ and $d_i(4)$, $i\in\Z_4$, form the $8$-cycle
\begin{equation}\label{eq:common-8-cycle}
v_0,u_0,v_1,u_1,v_2,u_2,v_3,u_3,v_0,
\end{equation}
so $X_\tau$ is connected. There are eight quotient types, and the corresponding graphs $X_\tau$ are displayed in Figure~\ref{fig:Q4p-quotient-graphs}, where parallel lines represent distinct parallel links.

\begingroup
\tikzset{
quotient vertex/.style={circle,fill=black,inner sep=1.2pt},
quotient edge/.style={line width=0.35pt},
quotient double/.style={quotient edge,double=white,double distance=1.1pt},
quotient triple/.style={quotient edge,double=white,double distance=2.2pt},
every label/.append style={font=\scriptsize,inner sep=1pt}
}
\newcommand{\quotientorbit}[2]{%
\ifcase#2\relax
\or
    \foreach \x/\y in {#1}{\draw[quotient edge] (\x)--(\y);}
\or
    \foreach \x/\y in {#1}{\draw[quotient double] (\x)--(\y);}
\or
    \foreach \x/\y in {#1}{\draw[quotient triple] (\x)--(\y);\draw[quotient edge] (\x)--(\y);}
\fi
}
\newcommand{\quotientgraph}[4]{%
\begin{tikzpicture}
\coordinate (u0) at (90:0.95);
\coordinate (v0) at (45:0.95);
\coordinate (u3) at (0:0.95);
\coordinate (v3) at (-45:0.95);
\coordinate (u2) at (-90:0.95);
\coordinate (v2) at (-135:0.95);
\coordinate (u1) at (180:0.95);
\coordinate (v1) at (135:0.95);
\quotientorbit{u0/v0,u1/v1,u2/v2,u3/v3}{#1}
\quotientorbit{u0/v1,u1/v2,u2/v3,u3/v0}{#2}
\quotientorbit{u0/v2,u1/v3,u2/v0,u3/v1}{#3}
\quotientorbit{u0/v3,u1/v0,u2/v1,u3/v2}{#4}
\node[quotient vertex,label=90:$u_0$] at (u0) {};
\node[quotient vertex,label=45:$v_0$] at (v0) {};
\node[quotient vertex,label=0:$u_3$] at (u3) {};
\node[quotient vertex,label=-45:$v_3$] at (v3) {};
\node[quotient vertex,label=-90:$u_2$] at (u2) {};
\node[quotient vertex,label=-135:$v_2$] at (v2) {};
\node[quotient vertex,label=180:$u_1$] at (u1) {};
\node[quotient vertex,label=135:$v_1$] at (v1) {};
\end{tikzpicture}%
}

\begin{figure}[htbp]
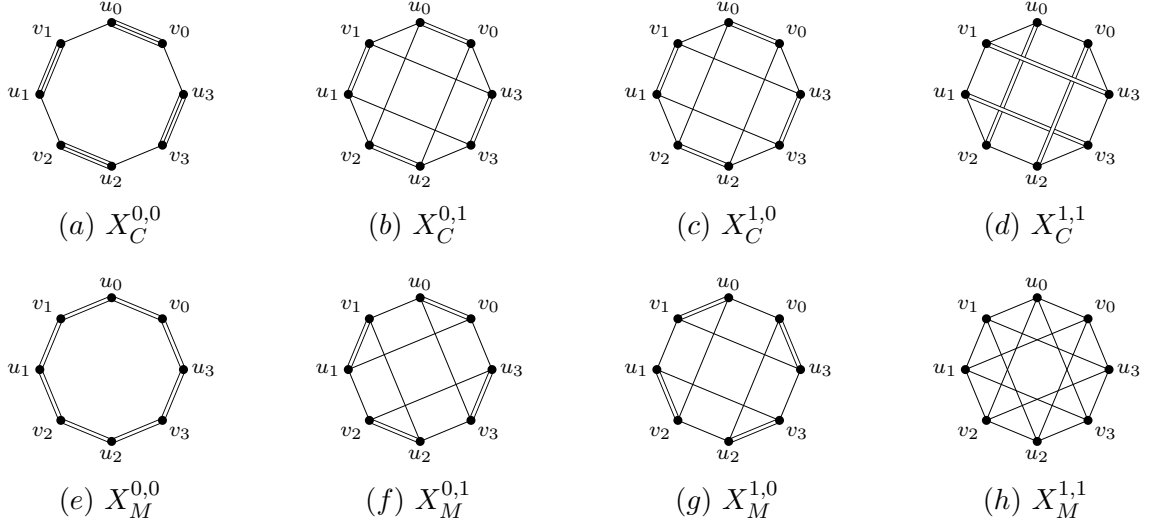

\centering
\begin{minipage}[t]{0.235\textwidth}
\centering
\quotientgraph{3}{1}{0}{0}\par\vspace{0.25em}
$(a)\ X_C^{0,0}$
\end{minipage}
\hfill
\begin{minipage}[t]{0.235\textwidth}
\centering
\quotientgraph{2}{1}{1}{0}\par\vspace{0.25em}
$(b)\ X_C^{0,1}$
\end{minipage}
\hfill
\begin{minipage}[t]{0.235\textwidth}
\centering
\quotientgraph{2}{1}{1}{0}\par\vspace{0.25em}
$(c)\ X_C^{1,0}$
\end{minipage}
\hfill
\begin{minipage}[t]{0.235\textwidth}
\centering
\quotientgraph{1}{1}{2}{0}\par\vspace{0.25em}
$(d)\ X_C^{1,1}$
\end{minipage}

\par\vspace{1em}

\begin{minipage}[t]{0.235\textwidth}
\centering
\quotientgraph{2}{2}{0}{0}\par\vspace{0.25em}
$(e)\ X_M^{0,0}$
\end{minipage}
\hfill
\begin{minipage}[t]{0.235\textwidth}
\centering
\quotientgraph{2}{1}{0}{1}\par\vspace{0.25em}
$(f)\ X_M^{0,1}$
\end{minipage}
\hfill
\begin{minipage}[t]{0.235\textwidth}
\centering
\quotientgraph{1}{2}{1}{0}\par\vspace{0.25em}
$(g)\ X_M^{1,0}$
\end{minipage}
\hfill
\begin{minipage}[t]{0.235\textwidth}
\centering
\quotientgraph{1}{1}{1}{1}\par\vspace{0.25em}
$(h)\ X_M^{1,1}$
\end{minipage}
\caption{The eight quotient graphs $X_F^{\epsilon,\delta}$.}
\label{fig:Q4p-quotient-graphs}
\end{figure}
\endgroup

By Table~\ref{tab:Q4p-quotient-voltages}, the voltage of each dart $d_i(j)$ has the form $c_1\kappa+c_2\lambda$, where the integers $c_1$ and $c_2$ depend only on $i$ and $j$, and not on $F$, $k$, $\ell$ or $p$. For example, $\zeta_\Gamma(d_i(2))=(-1)^i\kappa$, so that $(c_1,c_2)=((-1)^i,0)$. The following definition assigns to each dart its coefficient vector $(c_1,c_2)\in\Z^{1\times2}$.

\begin{definition}\label{def:integral-voltage}
Put $\mathbf e_1=(1,0)$, $\mathbf e_2=(0,1)$ and $\mathbf 0=(0,0)$ in $\Z^{1\times2}$. For a quotient type $\tau=(F,\epsilon,\delta)$, the \emph{integral voltage assignment} of $X_\tau$ is the map $\widetilde\zeta_\tau\colon D(X_\tau)\to\Z^{1\times2}$ determined, for $i\in\Z_4$ and $1\leq j\leq4$, by
\[
\widetilde\zeta_\tau\bigl(d_i(j)\bigr)=
\begin{cases}
(-1)^i\mathbf e_1,&j=2,\\
(-1)^i\mathbf e_2,&j=3,\\
\mathbf 0,&j\in\{1,4\},
\end{cases}
\qquad\text{and}\qquad
\widetilde\zeta_\tau\bigl(d_i(j)^{-1}\bigr)=-\widetilde\zeta_\tau\bigl(d_i(j)\bigr).
\]
This is well defined by~\eqref{eq:dart-decomposition}. As for ordinary voltages in Subsection~\ref{subsec:ordinary-voltages}, we extend $\widetilde\zeta_\tau$ to walks by putting $\widetilde\zeta_\tau(W)=\sum_{n=1}^{m}\widetilde\zeta_\tau(f_n)$ for every walk $W=(f_1,\ldots,f_m)$ in $X_\tau$.
\end{definition}

For a prime $p$, let $\pi_p\colon\Z\to\F_p=\Z/p\Z$ be the canonical homomorphism $n\mapsto n+p\Z$. If $B=(b_{ij})$ is an integer matrix, write
\[
B^{(p)}=\bigl(\pi_p(b_{ij})\bigr)
\]
for the corresponding matrix over $\F_p$.

\begin{lemma}\label{lem:integral-voltage}
Let $\Gamma=\Gamma_F(k,\ell)$ be connected of quotient type $\tau$, as in Notation~$\ref{not:Q4p-quotient}$. Then, for every walk $W$ in $X_\tau$,
\begin{equation}\label{eq:specialisation}
\zeta_\Gamma(W)=\widetilde\zeta_\tau(W)^{(p)}\begin{pmatrix}\kappa\\\lambda\end{pmatrix}.
\end{equation}
\end{lemma}

\begin{proof}
By Table~\ref{tab:Q4p-quotient-voltages} and Definition~\ref{def:integral-voltage}, the two sides of~\eqref{eq:specialisation} agree when $W$ consists of a single dart $d_i(j)$, and both sides change sign when this dart is replaced by its inverse. By~\eqref{eq:dart-decomposition}, they agree for every walk of length $1$. Since $\pi_p$ is a homomorphism, both sides are additive over the darts of a walk, and hence~\eqref{eq:specialisation} holds for every walk.
\end{proof}

Some of the quotient graphs are isomorphic, namely $X_C^{0,1}\cong X_C^{1,0}\cong X_C^{1,1}$ and $X_M^{0,1}\cong X_M^{1,0}$. The eight quotient types are nevertheless kept distinct, since $\tau$ determines the integral voltage assignment $\widetilde\zeta_\tau$ of Definition~\ref{def:integral-voltage} as well as the graph $X_\tau$.

\begin{lemma}\label{lem:Q4p-quotient-separation}
Let $p\geq5$. A connected graph in the $C$-family is not isomorphic to a connected graph in the $M$-family.
\end{lemma}

\begin{proof}
Suppose that a connected graph $\Gamma$ of the $C$-family is isomorphic to a connected graph $\Gamma'$ of the $M$-family, and use Notation~\ref{not:Q4p-quotient} for $\Gamma$ and $\Gamma'$. By Proposition~\ref{prop:sylow-alignment} there is an isomorphism $f\colon\Gamma\to\Gamma'$ with $f^{-1}P_\Gamma f=P_{\Gamma'}$, so $\Gamma_{P_\Gamma}\cong\Gamma'_{P_{\Gamma'}}$ by Lemma~\ref{lem:quotient-alignment}. Hence $X_\tau\cong X_{\tau'}$ for some quotient type $\tau$ of the $C$-family and some quotient type $\tau'$ of the $M$-family.

For a link $e$, call the number of links with the same endvertices as $e$, including $e$ itself, the \emph{multiplicity} of $e$. The following properties of the quotient graphs follow from Definition~\ref{def:quotient-type} and can be read off from Figure~\ref{fig:Q4p-quotient-graphs}. The maximum multiplicity of a link is $3$ in $X_C^{0,0}$ and at most $2$ in every quotient graph of the $M$-family. Each of $X_C^{0,1}$, $X_C^{1,0}$ and $X_C^{1,1}$ has links of multiplicity $1$ as well as links of multiplicity $2$, the latter form a perfect matching, and deleting them leaves an $8$-cycle. Among the quotient graphs of the $M$-family, every link of $X_M^{0,0}$ has multiplicity $2$, every link of $X_M^{1,1}=K_{4,4}$ has multiplicity $1$, and deleting from $X_M^{0,1}$ or from $X_M^{1,0}$ the perfect matching formed by the links of multiplicity $2$ leaves two disjoint $4$-cycles. All these properties are invariant under isomorphism, so $X_\tau\not\cong X_{\tau'}$, a contradiction.
\end{proof}

By Lemma~\ref{lem:Q4p-quotient-separation} it remains to compare graphs within one family. For the lifting computation, we choose in each quotient graph $X_\tau$ the spanning tree obtained by deleting the link containing $d_3(4)$ from the $8$-cycle~\eqref{eq:common-8-cycle}. The associated fundamental closed walks give a homology basis for that quotient graph, as shown in the following lemma.

\begin{lemma}\label{lem:Q4p-common-tree}
Let $\tau$ be a quotient type, and let $T$ be the spanning tree of $X_\tau$ obtained by deleting from the $8$-cycle~\eqref{eq:common-8-cycle} the link containing $d_3(4)$. Then $X_\tau$ has exactly nine cotree links, and $\widetilde\zeta_\tau(C_d)=\widetilde\zeta_\tau(d)$ for the fundamental closed walk $C_d=dQ_d$ determined by any cotree dart $d$, where $Q_d$ is as in Subsection~$\ref{subsec:fundamental}$. Moreover, for every prime $p$, the homology classes of these nine fundamental closed walks form a basis of $H_1(X_\tau;\F_p)$.
\end{lemma}

\begin{proof}
Deleting one link of the $8$-cycle~\eqref{eq:common-8-cycle} leaves a path through all eight vertices of $X_\tau$, that is, a spanning tree with seven links; since $X_\tau$ has sixteen links, nine of them lie outside $T$.

By Definition~\ref{def:integral-voltage}, the assignment $\widetilde\zeta_\tau$ vanishes on the darts of the links of the $8$-cycle~\eqref{eq:common-8-cycle}, and hence on every dart of $T$. For a cotree dart $d$, the walk $Q_d$ lies entirely in $T$, so $\widetilde\zeta_\tau(Q_d)=\mathbf 0$ and therefore $\widetilde\zeta_\tau(C_d)=\widetilde\zeta_\tau(d)$. Finally, $X_\tau$ has no semiedges, so by Subsection~\ref{subsec:fundamental} the homology classes of the nine fundamental closed walks form a basis of $H_1(X_\tau;\F_p)$ for every prime $p$.
\end{proof}

We now explain the purpose of the notation below. For two graphs $\Gamma$ and $\Gamma'$ of the same family and an isomorphism $\varphi$ between their quotient graphs, the lifting criterion of Corollary~\ref{cor:homology-comparison} need only be checked on the nine closed walks of Lemma~\ref{lem:Q4p-common-tree}, and by Lemma~\ref{lem:integral-voltage} every voltage involved is an integral linear combination of $\kappa,\lambda$ or of $\kappa',\lambda'$. Hence the lifting condition is a system of nine linear equations whose coefficients are integers depending only on the quotient types and on $\varphi$, and not on $p$, $k$ or $\ell$. Notation~\ref{not:Q4p-matrices} collects these coefficients in one integer matrix $L_\varphi\in\Z^{9\times4}$. It is assembled from three simpler matrices: $W_\tau$ lists the integral voltages of the sixteen darts, $H_\tau$ records, with signs, how often each of the nine closed walks traverses each dart, and $P_\varphi$ records how $\varphi$ permutes the darts. Since only finitely many matrices $L_\varphi$ occur, the lifting problem for all graphs in the two families and all primes $p\geq5$ reduces to a finite computation with these matrices.

\begin{notation}\label{not:Q4p-matrices}
Let $\tau$ and $\tau'$ be quotient types of the same family, and let $\varphi\colon X_\tau\to X_{\tau'}$ be an isomorphism between the corresponding quotient graphs of Definition~\ref{def:quotient-type}. By~\eqref{eq:dart-decomposition}, every link of $X_\tau$ contains exactly one of the sixteen darts $d_i(j)$, which are indexed by the pairs $(i,j)$ with $i\in\Z_4$ and $1\leq j\leq4$; here $d_i(j)$ is the dart beginning at $u_i$ that corresponds to the element $s_j$ of the connection set (Notation~\ref{not:Q4p-quotient} and Definition~\ref{def:quotient-type}). We list these darts in a fixed order, the same for all quotient types: choose an ordering $(i_1,j_1),\ldots,(i_{16},j_{16})$ of the sixteen index pairs, and put $\bar d_n=d_{i_n}(j_n)$ in $X_\tau$ and $\bar d'_n=d_{i_n}(j_n)$ in $X_{\tau'}$ for $1\leq n\leq16$. For a walk $C$ in $X_\tau$, let $h_\tau(C)\in\Z^{1\times16}$ be the row vector whose $n$th entry is the number of occurrences of $\bar d_n$ in $C$ minus the number of occurrences of $\bar d_n^{-1}$ in $C$, and define $h_{\tau'}$ on walks in $X_{\tau'}$ in the same way. Let $C_1,\ldots,C_9$ be the fundamental closed walks $C_d$ of Lemma~\ref{lem:Q4p-common-tree}, where $d$ runs, in a fixed order, through the darts $d_i(j)$ lying in the nine cotree links. With $\widetilde\zeta_\tau$ and $\widetilde\zeta_{\tau'}$ the integral voltage assignments of Definition~\ref{def:integral-voltage}, put
\[
W_\tau=\begin{pmatrix}\widetilde\zeta_\tau(\bar d_1)\\ \vdots\\ \widetilde\zeta_\tau(\bar d_{16})\end{pmatrix}\in\Z^{16\times2},
\quad
H_\tau=\begin{pmatrix}h_\tau(C_1)\\ \vdots\\ h_\tau(C_9)\end{pmatrix}\in\Z^{9\times16},
\]
\[
P_\varphi=(s_{nm})\in\Z^{16\times16}, \ \text{ where }\
s_{nm}=
\begin{cases}
1,&\bar d_n^\varphi=\bar d'_m,\\
-1,&\bar d_n^\varphi=(\bar d'_m)^{-1},\\
0,&\text{otherwise},
\end{cases}
\]
and
\[
L_\varphi=\begin{pmatrix}H_\tau P_\varphi W_{\tau'}&-H_\tau W_\tau\end{pmatrix}\in\Z^{9\times4}.
\]
Thus $W_\tau$ and $H_\tau$ depend only on $\tau$ and on the chosen enumeration, and $P_\varphi$ depends only on $\varphi$ and on the chosen enumeration; in particular, these matrices are independent of $p$, $k$ and $\ell$.
\end{notation}

\begin{lemma}\label{lem:Q4p-rows}
In Notation~$\ref{not:Q4p-matrices}$, $P_\varphi$ is a signed permutation matrix, and for $1\leq r\leq9$ the $r$th rows of $H_\tau W_\tau$ and of $H_\tau P_\varphi W_{\tau'}$ are $\widetilde\zeta_\tau(C_r)$ and $\widetilde\zeta_{\tau'}(C_r^\varphi)$, respectively.
\end{lemma}

\begin{proof}
As $\varphi$ induces a bijection between the links of $X_\tau$ and those of $X_{\tau'}$, each row and each column of $P_\varphi$ has exactly one nonzero entry, and so $P_\varphi$ is a signed permutation matrix.

Let $C$ be a walk in $X_\tau$. By Definition~\ref{def:integral-voltage}, $\widetilde\zeta_\tau$ is additive on walks and $\widetilde\zeta_\tau(d^{-1})=-\widetilde\zeta_\tau(d)$ for every dart $d$, so $\widetilde\zeta_\tau(C)=h_\tau(C)W_\tau$; likewise $\widetilde\zeta_{\tau'}(C')=h_{\tau'}(C')W_{\tau'}$ for every walk $C'$ in $X_{\tau'}$. Moreover $\varphi$ carries each occurrence of $\bar d_n^{\pm1}$ in $C$ to an occurrence of $(\bar d'_m)^{\pm s_{nm}}$ in $C^\varphi$, where $m$ is the index with $s_{nm}\neq0$; summing over $n$ gives $h_{\tau'}(C^\varphi)=h_\tau(C)P_\varphi$. Taking $C=C_r$ and using that the $r$th row of $H_\tau$ is $h_\tau(C_r)$ gives both assertions.
\end{proof}

\begin{lemma}\label{lem:Q4p-lifting-matrix}
Under Notations~$\ref{not:Q4p-quotient}$ and~$\ref{not:Q4p-matrices}$, suppose that $\Gamma$ and $\Gamma'$ have quotient types $\tau$ and $\tau'$, respectively. Then $\varphi$ lifts along $\wp_{P_\Gamma}$ and $\wp_{P_{\Gamma'}}$ if and only if there exists $q\in\F_p^*$ such that
\[
L_\varphi^{(p)}
\begin{pmatrix}
\kappa'\\
\lambda'\\
q\kappa\\
q\lambda
\end{pmatrix}
=0.
\]
\end{lemma}

\begin{proof}
By Lemma~\ref{lem:Q4p-quotient-voltage}, $\zeta_\Gamma$ and $\zeta_{\Gamma'}$ are $\F_p$-voltage descriptions of $\wp_{P_\Gamma}$ and $\wp_{P_{\Gamma'}}$. Composing with the corresponding equivalences of covering projections, as in the proof of Corollary~\ref{cor:quotient-lifting-test}, shows that $\varphi$ lifts along $\wp_{P_\Gamma}$ and $\wp_{P_{\Gamma'}}$ if and only if it lifts along $\wp_{\zeta_\Gamma}$ and $\wp_{\zeta_{\Gamma'}}$. Every $\F_p$-linear automorphism of $\F_p$ is of the form $x\mapsto qx$ for a unique $q\in\F_p^*$. Hence, by Corollary~\ref{cor:homology-comparison}, $\varphi$ lifts if and only if there exists $q\in\F_p^*$ such that
\[
\overline\zeta_{\Gamma'}\bigl(\gamma^{\varphi_*}\bigr)=q\,\overline\zeta_\Gamma(\gamma)
\]
for every $\gamma\in H_1(X_\tau;\F_p)$. Since the homology classes represented by $C_1,\ldots,C_9$ form a basis of $H_1(X_\tau;\F_p)$ by Lemma~\ref{lem:Q4p-common-tree}, this is equivalent to
\[
\zeta_{\Gamma'}(C_r^\varphi)=q\,\zeta_\Gamma(C_r)
\]
for $1\leq r\leq9$. By Lemmas~\ref{lem:integral-voltage} and~\ref{lem:Q4p-rows}, $\zeta_{\Gamma'}(C_r^\varphi)$ and $\zeta_\Gamma(C_r)$ are the $r$th entries of
\[
(H_\tau P_\varphi W_{\tau'})^{(p)}
\begin{pmatrix}
\kappa'\\
\lambda'
\end{pmatrix}
\quad\text{and}\quad
(H_\tau W_\tau)^{(p)}
\begin{pmatrix}
\kappa\\
\lambda
\end{pmatrix},
\]
respectively. By the definition of $L_\varphi$, the nine equalities are therefore equivalent to the displayed condition.
\end{proof}

For a field $K$ and a matrix $B$ with $n$ columns and entries in $K$, we regard the elements of $K^n$ as column vectors and write
\[
\ker_K B=\{x\in K^n\mid Bx=0\},
\]
and $\operatorname{Row}_K(B)$ for the subspace of $K^{1\times n}$ spanned by the rows of $B$. Thus $\ker_K B=\{x\in K^n\mid wx=0\text{ for all }w\in\operatorname{Row}_K(B)\}$. We shall use the following elementary consequence of the Smith normal form.

\begin{lemma}\label{lem:kernel-mod-p}
Let $A\in\Z^{m\times n}$ and $D\in\Z^{s\times n}$, and suppose that $\ker_{\Q}A\leq\ker_{\Q}D$. If none of the nonzero diagonal entries in a Smith normal form of $A$ is divisible by $p$, then
\[
\ker_{\F_p}A^{(p)}\leq\ker_{\F_p}D^{(p)}.
\]
\end{lemma}

\begin{proof}
Let $UAV=\diag(g_1,\ldots,g_r,0,\ldots,0)$ be a Smith normal form of $A$, where $U$ and $V$ are unimodular and $g_1,\ldots,g_r$ are positive integers. Taking orthogonal complements in $\Q^n$, the hypothesis $\ker_{\Q}A\leq\ker_{\Q}D$ gives $\operatorname{Row}_{\Q}(D)\leq\operatorname{Row}_{\Q}(A)$. Let $w$ be a row of $D$. Then $wV\in\operatorname{Row}_{\Q}(AV)=\operatorname{Row}_{\Q}(UAV)$ and $wV\in\Z^n$, so
\[
wV=(c_1,\ldots,c_r,0,\ldots,0)\ \text{ for some }\ c_1,\ldots,c_r\in\Z.
\]
Since $p\nmid g_i$ for $1\leq i\leq r$, the row space of $(UAV)^{(p)}=\diag\bigl(\pi_p(g_1),\ldots,\pi_p(g_r),0,\ldots,0\bigr)$ consists of the vectors in $\F_p^n$ whose last $n-r$ coordinates are zero; hence $w^{(p)}V^{(p)}=(wV)^{(p)}\in\operatorname{Row}_{\F_p}\bigl((UAV)^{(p)}\bigr)$. Since $\det U,\det V\in\{1,-1\}$, the matrices $U^{(p)}$ and $V^{(p)}$ are invertible, so $\operatorname{Row}_{\F_p}\bigl((UAV)^{(p)}\bigr)=\operatorname{Row}_{\F_p}\bigl(A^{(p)}\bigr)V^{(p)}$, and therefore $w^{(p)}\in\operatorname{Row}_{\F_p}\bigl(A^{(p)}\bigr)$. Thus $\operatorname{Row}_{\F_p}\bigl(D^{(p)}\bigr)\leq\operatorname{Row}_{\F_p}\bigl(A^{(p)}\bigr)$, and taking orthogonal complements in $\F_p^n$ gives $\ker_{\F_p}A^{(p)}\leq\ker_{\F_p}D^{(p)}$.
\end{proof}

There are finitely many ordered pairs $(\tau,\tau')$ of quotient types from the same family and, for each pair, finitely many isomorphisms $\varphi\colon X_\tau\to X_{\tau'}$; the matrices $L_\varphi$ are therefore finite in number, and their entries are integers independent of $p$, $k$ and $\ell$. For each of them, Appendix~\ref{apdx:pge5}\ref{apdx:pge5-C} and~\ref{apdx:pge5-M} exhibit a matrix $D_\varphi\in\Z^{2\times4}$ with
\[
\ker_{\Q}L_\varphi\leq\ker_{\Q}D_\varphi,
\]
and Appendix~\ref{apdx:pge5}\ref{apdx:pge5-Smith} shows that the nonzero Smith invariants of $L_\varphi$ all lie in $\{1,2,3\}$, hence are divisible by no prime $p\geq5$. Lemma~\ref{lem:kernel-mod-p}, applied with $A=L_\varphi$ and $D=D_\varphi$, therefore gives $\ker_{\F_p}L_\varphi^{(p)}\leq\ker_{\F_p}D_\varphi^{(p)}$ for every $p\geq5$. These containments, together with the parity relations also verified in Appendix~\ref{apdx:pge5}, give the following lemma, in which $(e_0,e_1,e_2)$ is as in~\eqref{eq:e-and-x} and
\[
(e'_0,e'_1,e'_2)=(0,k',\ell'),\qquad (x_0,x_1,x_2)=(0,\kappa,\lambda),\qquad (x'_0,x'_1,x'_2)=(0,\kappa',\lambda').
\]

\begin{lemma}\label{lem:Q4p-finite-computation}
Under Notation~$\ref{not:Q4p-quotient}$, suppose that $\Gamma=\Gamma_F(k,\ell)$ and $\Gamma'=\Gamma_F(k',\ell')$ for one and the same $F\in\{C,M\}$, and that $\Gamma\cong\Gamma'$.
\begin{enumerate}[label=\textup{(\roman*)},ref=\textup{(\roman*)}]
\item\label{lem:Q4p-finite-computation-C}
If $F=C$, then there exist $q_0\in\F_p^*$ and $\rho\in S_3$ such that, for all $0\leq i,j\leq2$,
\[
x'_i-x'_j=q_0\bigl(x_{\rho(i)}-x_{\rho(j)}\bigr) \ \text{ and }\ e'_i-e'_j\equiv e_{\rho(i)}-e_{\rho(j)}\pmod2.
\]

\item\label{lem:Q4p-finite-computation-M}
If $F=M$, then there exist $q_0\in\F_p^*$ and $\sigma\in\{1,-1\}$ such that
\[
\begin{aligned}
\text{either }\ &(\kappa',\lambda')=q_0(\kappa,\sigma\lambda)\ \text{and}\ (\epsilon_{k'},\epsilon_{\ell'})=(\epsilon_k,\epsilon_\ell),\\
\text{or }\ &(\kappa',\lambda')=q_0(\lambda,\sigma\kappa)\ \text{and}\ (\epsilon_{k'},\epsilon_{\ell'})=(\epsilon_\ell,\epsilon_k).
\end{aligned}
\]
\end{enumerate}
\end{lemma}

\begin{proof}
By Proposition~\ref{prop:sylow-alignment} and Lemma~\ref{lem:quotient-alignment}, there is an isomorphism $\varphi\colon X_{\tau(\Gamma)}\to X_{\tau(\Gamma')}$ that lifts along $\wp_{P_\Gamma}$ and $\wp_{P_{\Gamma'}}$. Hence, by Lemma~\ref{lem:Q4p-lifting-matrix}, there exists $q\in\F_p^*$ such that, on putting $u=q\kappa$ and $v=q\lambda$,
\[
(\kappa',\lambda',u,v)^{\mathsf T}\in\ker_{\F_p}L_\varphi^{(p)}.
\]
By the preceding paragraph, $\ker_{\F_p}L_\varphi^{(p)}\leq\ker_{\F_p}D_\varphi^{(p)}$. Hence $(r',s',u,v)=(\kappa',\lambda',u,v)$ satisfies, over $\F_p$, the relations recorded in Appendix~\ref{apdx:pge5}\ref{apdx:pge5-C} and~\ref{apdx:pge5-M}. Moreover, $\tau(\Gamma)=(F,\epsilon_k,\epsilon_\ell)$ and $\tau(\Gamma')=(F,\epsilon_{k'},\epsilon_{\ell'})$, so the parities $\epsilon,\delta,\epsilon',\delta'$ in the Appendix are $\epsilon_k,\epsilon_\ell,\epsilon_{k'},\epsilon_{\ell'}$.

Suppose that $F=C$. By Appendix~\ref{apdx:pge5}\ref{apdx:pge5-C}, there exist $\rho\in S_3$ and $\eta\in\{1,-1\}$ such that
\[
(\kappa',\lambda')=\eta\bigl(y_{\rho(1)}-y_{\rho(0)},\,y_{\rho(2)}-y_{\rho(0)}\bigr),\ \text{ where }\ (y_0,y_1,y_2)=(0,u,v),
\]
and $e'_i-e'_j\equiv e_{\rho(i)}-e_{\rho(j)}\pmod2$ for all $0\leq i,j\leq2$. Since $(y_0,y_1,y_2)=q(x_0,x_1,x_2)$, putting $q_0=\eta q$ gives $x'_i=q_0\bigl(x_{\rho(i)}-x_{\rho(0)}\bigr)$ for $i=1,2$; this also holds for $i=0$, since both sides are then $0$. Subtracting two of these three relations gives $x'_i-x'_j=q_0\bigl(x_{\rho(i)}-x_{\rho(j)}\bigr)$ for all $0\leq i,j\leq2$. This proves~\ref{lem:Q4p-finite-computation-C}.

Suppose that $F=M$. By Appendix~\ref{apdx:pge5}\ref{apdx:pge5-M}, for some $\eta,\sigma\in\{1,-1\}$ we have
\begin{align*}
\text{either }\ &(\kappa',\lambda')=\eta(u,\sigma v),\ \ (\epsilon_{k'},\epsilon_{\ell'})=(\epsilon_k,\epsilon_\ell),\\
\text{or }\ &(\kappa',\lambda')=\eta(v,\sigma u),\ \ (\epsilon_{k'},\epsilon_{\ell'})=(\epsilon_\ell,\epsilon_k).
\end{align*}
Since $(u,v)=q(\kappa,\lambda)$, statement~\ref{lem:Q4p-finite-computation-M} follows on putting $q_0=\eta q$.
\end{proof}

Lemma~\ref{lem:Q4p-finite-computation} determines the parameters modulo $p$ and modulo $2$; the Chinese remainder theorem recovers them in $\Z_{2p}$.

\begin{lemma}\label{lem:Q4p-recovery}
Suppose that $p\geq5$. Then the following statements hold.
\begin{enumerate}[label=\textup{(\roman*)},ref=\textup{(\roman*)}]
\item\label{lem:Q4p-recovery-C}
If $\Gamma_C(k,\ell)\cong\Gamma_C(k',\ell')$, then $(k',\ell')\in\mathcal O_C(k,\ell)$.
\item\label{lem:Q4p-recovery-M}
If $\Gamma_M(k,\ell)$ and $\Gamma_M(k',\ell')$ are connected and isomorphic, then $(k',\ell')\in\mathcal O_M(k,\ell)$.
\end{enumerate}
\end{lemma}

\begin{proof}
We use Notation~\ref{not:Q4p-quotient} throughout. For statement~\ref{lem:Q4p-recovery-C}, note that both graphs are connected by Lemma~\ref{lem:Q4p-connectivity}, and let $q_0$ and $\rho$ be as in Lemma~\ref{lem:Q4p-finite-computation}\ref{lem:Q4p-finite-computation-C}. By the Chinese remainder theorem, there is a unique $r\in\Z_{2p}$ with $r\equiv q_0\pmod p$ and $r\equiv1\pmod2$, and $r\in\Z_{2p}^*$ because $r$ is odd and $q_0\neq0$. Since $2x_i\equiv e_i\pmod p$ for $0\leq i\leq2$, multiplying the first relation of Lemma~\ref{lem:Q4p-finite-computation}\ref{lem:Q4p-finite-computation-C} by $2$ gives $e'_i-e'_j\equiv r\bigl(e_{\rho(i)}-e_{\rho(j)}\bigr)\pmod p$, while the second relation together with $r\equiv1\pmod2$ gives $e'_i-e'_j\equiv r\bigl(e_{\rho(i)}-e_{\rho(j)}\bigr)\pmod2$. By the Chinese remainder theorem,
\[
e'_i-e'_j= r\bigl(e_{\rho(i)}-e_{\rho(j)}\bigr)\ \text{ in }\ \Z_{2p}.
\]
Taking $(i,j)=(1,0)$ and $(2,0)$ and using $e'_0=0$ gives
\[
(k',\ell')=r\bigl(e_{\rho(1)}-e_{\rho(0)},\,e_{\rho(2)}-e_{\rho(0)}\bigr)\in\mathcal O_C(k,\ell).
\]

For statement~\ref{lem:Q4p-recovery-M}, let $q_0$ and $\sigma$ be as in Lemma~\ref{lem:Q4p-finite-computation}\ref{lem:Q4p-finite-computation-M}, and choose $r\in\Z_{2p}^*$ as above. Suppose first that the first alternative in Lemma~\ref{lem:Q4p-finite-computation}\ref{lem:Q4p-finite-computation-M} holds. From $\kappa'=q_0\kappa$ and $\lambda'=q_0\sigma\lambda$ we obtain, after multiplying by $2$,
\[
k'\equiv rk\pmod p\ \text{ and }\ \ell'\equiv r\sigma\ell\pmod p .
\]
Since $r$ and $\sigma$ are odd, $r\sigma\in\Z_{2p}^*$ and $r\sigma\equiv1\pmod2$; hence $(\epsilon_{k'},\epsilon_{\ell'})=(\epsilon_k,\epsilon_\ell)$ gives $k'\equiv rk\pmod 2$ and $\ell'\equiv r\sigma\ell\pmod2$. By the Chinese remainder theorem, $(k',\ell')=r(k,\sigma\ell)$ in $\Z_{2p}$, and therefore $(k',\ell')\in\mathcal O_M(k,\ell)$. In the second alternative, $(\kappa',\lambda')=q_0(\lambda,\sigma\kappa)$ and $(\epsilon_{k'},\epsilon_{\ell'})=(\epsilon_\ell,\epsilon_k)$, and the same argument with $(k,\ell)$ replaced by $(\ell,k)$ gives $(k',\ell')=r(\ell,\sigma k)\in\mathcal O_M(k,\ell)$.
\end{proof}

\begin{proof}[Proof of Proposition~$\ref{prop:Q4p-classification}$]
The reverse implications in~\ref{prop:Q4p-classification-C} and~\ref{prop:Q4p-classification-M} follow from Lemma~\ref{lem:Q4p-evident-transformations}.

For $p=3$, the remaining assertions are verified by the finite computation described in Appendix~\ref{apdx:p3}. Suppose now that $p\geq5$. Statement~\ref{prop:Q4p-classification-CM} holds: by Lemma~\ref{lem:Q4p-connectivity} every graph of the $C$-family is connected, so by Lemma~\ref{lem:Q4p-quotient-separation} it is not isomorphic to a connected graph of the $M$-family, nor to the disconnected graph $\Gamma_M(p,p)$. By Lemma~\ref{lem:Q4p-recovery}\ref{lem:Q4p-recovery-C}, the forward implication in~\ref{prop:Q4p-classification-C} holds, and by Lemma~\ref{lem:Q4p-recovery}\ref{lem:Q4p-recovery-M} the forward implication in~\ref{prop:Q4p-classification-M} holds whenever both graphs are connected.

By Lemma~\ref{lem:Q4p-connectivity}, the only disconnected graph in the $M$-family is $\Gamma_M(p,p)$, so $\Gamma_M(p,p)\cong\Gamma_M(k',\ell')$ forces $(k',\ell')=(p,p)$. Since $rp=p$ in $\Z_{2p}$ for every odd $r$, and since $-p=p$, we have $\mathcal O_M(p,p)=\{(p,p)\}$, and the forward implication in~\ref{prop:Q4p-classification-M} also holds in the disconnected case. Thus~\ref{prop:Q4p-classification-C}--\ref{prop:Q4p-classification-CM} are proved.

Finally, if two graphs in~\ref{prop:Q4p-classification-C} or~\ref{prop:Q4p-classification-M} are isomorphic, the corresponding forward implication shows that their parameter pairs belong to the same set $\mathcal O_C(k,\ell)$ or $\mathcal O_M(k,\ell)$. Lemma~\ref{lem:Q4p-evident-transformations} then shows that their connection sets are BCI-equivalent.
\end{proof}

We now return to arbitrary four-element connection sets.

\begin{lemma}\label{lem:Q4p-four-set-normal-form}
Let $A=\langle a\rangle$, and let $S\subseteq Q$ with $|S|=4$ and $1\in S$.
\begin{enumerate}[label=\textup{(\roman*)}]
\item If $|S\cap A|=2$, then $S$ is BCI-equivalent to $M(k,\ell)$ for some $(k,\ell)\in\mathcal P_M$.
\item If $|S\cap A|\in\{1,3\}$, then $S$ is BCI-equivalent to $C(k,\ell)$ for some $(k,\ell)\in\mathcal P_C$.
\end{enumerate}
\end{lemma}

\begin{proof}
Suppose first that $|S\cap A|=2$. Then $S=\{1,a^k,ba^i,ba^j\}$ with $k\neq0$ and $i\neq j$. Hence $S^{\sigma_{1,-j}}=M(k,i-j)$, where $(k,i-j)\in\mathcal P_M$.

Suppose next that $|S\cap A|=3$. Then $S=\{1,a^k,a^\ell,ba^t\}$ with $k,\ell\neq0$ and $k\neq\ell$. Hence $S^{\sigma_{1,-t}}=C(k,\ell)$, where $(k,\ell)\in\mathcal P_C$.

Finally, suppose that $|S\cap A|=1$, say $S=\{1,ba^{i_1},ba^{i_2},ba^{i_3}\}$ with $i_1,i_2,i_3$ pairwise distinct. Applying $\sigma_{1,-i_3}$, we may assume that $S=\{1,ba^i,ba^j,b\}$ with $i,j\neq0$ and $i\neq j$. Since $b^{-1}=ba^p$, we have $(b^{-1}S)^{\sigma_{1,p}}=C(i,j)$, where $(i,j)\in\mathcal P_C$.
\end{proof}

\begin{proof}[Proof of Theorem~$\ref{thm:intro-Q4p-4BCI}$]
The case $p=2$ follows from \cite[Lemma~3.4]{JinLiu2013}, so let $p$ be odd. By \cite[Theorem~1.1]{XieZhang2026}, $Q$ is a $3$-BCI-group; since isomorphic bi-Cayley graphs have the same valency, it therefore suffices to consider $S,T\subseteq Q$ with $|S|=|T|=4$ and $\BCay(Q,S)\cong\BCay(Q,T)$. After left multiplication, we may assume that $1\in S\cap T$. By Lemma~\ref{lem:bcay-components}, every connected component of $\BCay(Q,S)$ has $2|\langle S\rangle|$ vertices, and similarly for $T$; hence $|\langle S\rangle|=|\langle T\rangle|$.

Put $A=\langle a\rangle$. Suppose first that $S\subseteq A$. Then $T\subseteq A$, since otherwise $\langle T\rangle$ would contain an element $ba^i$ of order $4$, whence $4\mid|\langle T\rangle|$, whereas $|\langle S\rangle|$ divides $|A|=2p$ and $p$ is odd. As $A$ is cyclic, it has a unique subgroup of each order, so $\langle S\rangle=\langle T\rangle=:H$. By Lemma~\ref{lem:bcay-components}, the components of $\BCay(Q,S)$ and of $\BCay(Q,T)$ are isomorphic to $\BCay(H,S)$ and $\BCay(H,T)$, respectively, whence $\BCay(H,S)\cong\BCay(H,T)$. Now $|H|\geq|S|=4$ and $|H|$ divides $2p$ with $p$ odd, so $8\nmid|H|$ and $H\in\{\langle a^2\rangle,A\}$. By the classification of tetravalent cyclic Haar graphs \cite[Remark~1.2]{KoikeKovacs2014}, together with \cite[Proposition~2.7(iv)]{ABGR2020}, the cyclic group $H$ is a $4$-BCI-group; thus $T=gS^\delta$ for some $g\in H$ and $\delta\in\Aut H$. If $H=A$, then every $\delta\in\Aut H$ is the restriction of some $\sigma_{r,0}$ by~\eqref{eq:Q4p-automorphisms}. If $H=\langle a^2\rangle$, write $(a^2)^\delta=a^{2s}$ with $s\in\Z_p^*$ and choose $r\in\Z_{2p}^*$ such that $r\equiv s\pmod p$; then $\sigma_{r,0}|_H=\delta$. In either case $\delta$ extends to an automorphism of $Q$, and $S$ and $T$ are BCI-equivalent in $Q$.

Suppose now that $S\not\subseteq A$. Then $4\mid|\langle S\rangle|$, and hence $T\not\subseteq A$. By Lemma~\ref{lem:Q4p-four-set-normal-form}, there exist sets $S_0$ and $T_0$, each belonging to either the $C$-family or the $M$-family, such that $S$ is BCI-equivalent to $S_0$ and $T$ is BCI-equivalent to $T_0$. Hence $\BCay(Q,S_0)\cong\BCay(Q,T_0)$. By Proposition~\ref{prop:Q4p-classification}\ref{prop:Q4p-classification-CM}, $S_0$ and $T_0$ belong to the same family. The final assertion of Proposition~\ref{prop:Q4p-classification} then shows that $S_0$ and $T_0$ are BCI-equivalent, and therefore $S$ and $T$ are BCI-equivalent.

Consequently, every subset of $Q$ of cardinality at most $4$ is a BCI-subset, and $Q_{4p}$ is a $4$-BCI-group.
\end{proof}

\begin{remark}\label{rem:Q4p-classification}
The proof of Theorem~\ref{thm:intro-Q4p-4BCI} also yields an isomorphism classification of tetravalent bi-Cayley graphs over $Q_{4p}$. We do not list the resulting isomorphism classes or describe the graphs in each class, since this is long and not needed for the proof.
\end{remark}

\appendix

\section{Computational verification}
\label{apdx:computational-verification}

The Magma files used in the finite computations are supplied as ancillary files. We record here the precise assertions verified by them.

\subsection{The computation for \texorpdfstring{$p\geq5$}{p >= 5}}
\label{apdx:pge5}

For every ordered pair $\tau=(F,\epsilon,\delta)$ and $\tau'=(F,\epsilon',\delta')$ of quotient types from the same family $F$, the program enumerates all isomorphisms $\varphi\colon X_\tau\to X_{\tau'}$ in the dart model. For each such $\varphi$, it constructs
\[
L_\varphi=
\begin{pmatrix}
H_\tau P_\varphi W_{\tau'}&-H_\tau W_\tau
\end{pmatrix}
\in\Z^{9\times4}
\]
as in Notation~\ref{not:Q4p-matrices}, and verifies the assertions~\ref{apdx:pge5-Smith}--\ref{apdx:pge5-M} below. There a vector of $\Q^4$ is written $(r',s',u,v)$, and the subset of $\Q^4$ displayed in~\ref{apdx:pge5-C} or in~\ref{apdx:pge5-M} is the solution set of two linear equations with integer coefficients in $r',s',u,v$; we let $D_\varphi\in\Z^{2\times4}$ be the coefficient matrix of those two equations, so that the subset in question is $\ker_{\Q}D_\varphi$ and the containment asserted in~\ref{apdx:pge5-C} or in~\ref{apdx:pge5-M} reads
\[
\ker_{\Q}L_\varphi\leq\ker_{\Q}D_\varphi.
\]
In the first alternative of~\ref{apdx:pge5-M}, for instance, $D_\varphi=\left(\begin{smallmatrix}1&0&-\eta&0\\0&1&0&-\eta\sigma\end{smallmatrix}\right)$.

\begin{enumerate}[label=\textup{(\roman*)},ref=\textup{(\roman*)}]
\item\label{apdx:pge5-Smith} The four diagonal entries in a Smith normal form of $L_\varphi$ form one of
\[
(1,1,0,0),\ (1,1,1,0),\ (1,1,2,0),\ (1,1,1,1),\ (1,1,1,2),\ (1,1,1,3),\ (1,1,2,2).
\]

\item\label{apdx:pge5-C} Suppose that $F=C$. Put $(\bar e_0,\bar e_1,\bar e_2)=(0,\epsilon,\delta)$ and $(\bar e'_0,\bar e'_1,\bar e'_2)=(0,\epsilon',\delta')$. Then there exist $\rho\in S_3$ and $\eta\in\{1,-1\}$ such that
\[
\ker_{\Q}L_\varphi\leq
\left\{(r',s',u,v)\in\Q^4\mid
(r',s')=\eta\bigl(y_{\rho(1)}-y_{\rho(0)},y_{\rho(2)}-y_{\rho(0)}\bigr)\right\},
\]
where $(y_0,y_1,y_2)=(0,u,v)$, and
\[
\bar e'_i-\bar e'_j=\bar e_{\rho(i)}-\bar e_{\rho(j)}
\ \text{in }\F_2
\]
for all $0\leq i,j\leq2$.

\item\label{apdx:pge5-M} Suppose that $F=M$. Then, for some $\eta,\sigma\in\{1,-1\}$, 
\begin{align*}
\text{either }\ &\ker_{\Q}L_\varphi\leq
\{(r',s',u,v)\in\Q^4\mid(r',s')=\eta(u,\sigma v)\}\ \text{ and }\ 
(\epsilon',\delta')=(\epsilon,\delta),\\
\text{or }\ &\ker_{\Q}L_\varphi\leq
\{(r',s',u,v)\in\Q^4\mid(r',s')=\eta(v,\sigma u)\} \ \text{ and }\ 
(\epsilon',\delta')=(\delta,\epsilon).
\end{align*}
\end{enumerate}

\subsection{The computation for \texorpdfstring{$p=3$}{p = 3}}
\label{apdx:p3}

For $p=3$, the program first determines all distinct sets $\mathcal O_C(k,\ell)$ and $\mathcal O_M(k,\ell)$. It finds exactly three distinct sets in the $C$-family and exactly six distinct sets in the $M$-family, namely
\[
\begin{gathered}
\mathcal O_C(1,2),\ \mathcal O_C(1,3),\ \mathcal O_C(2,4),\\
\mathcal O_M(1,1),\ \mathcal O_M(1,2),\ \mathcal O_M(1,3),\
\mathcal O_M(2,2),\ \mathcal O_M(2,3),\ \mathcal O_M(3,3).
\end{gathered}
\]

By Lemma~$\ref{lem:Q4p-evident-transformations}$, parameter pairs belonging to the same set $\mathcal O_C(k,\ell)$ or $\mathcal O_M(k,\ell)$ give BCI-equivalent connection sets, and hence isomorphic graphs. It therefore remains to distinguish the graphs corresponding to the nine distinct sets above. 

For a graph $\Lambda$, let $A_\Lambda$ denote its adjacency matrix, let $\nu_0(A_\Lambda)$ denote the nullity of $A_\Lambda$ over $\Q$, and put
\[
I(\Lambda)=\bigl(\nu_0(A_\Lambda),\operatorname{tr}(A_\Lambda^4)\bigr).
\]
If $\Lambda_1\cong\Lambda_2$, then $A_{\Lambda_2}=P^{-1}A_{\Lambda_1}P$ for some permutation matrix $P$. Thus their adjacency matrices have the same rank and nullity, and $A_{\Lambda_2}^4=P^{-1}A_{\Lambda_1}^4P$ gives $\operatorname{tr}(A_{\Lambda_2}^4)=\operatorname{tr}(A_{\Lambda_1}^4)$. Hence $I(\Lambda)$ is an isomorphism invariant. Computing $I(\Lambda)$ for $\Lambda=\Gamma_F(k,\ell)$, with one parameter pair $(k,\ell)$ from each of the nine sets, gives
\[
\begin{array}{c|ccc|cccccc}
F&C&C&C&M&M&M&M&M&M\\
(k,\ell)&(1,2)&(1,3)&(2,4)&(1,1)&(1,2)&(1,3)&(2,2)&(2,3)&(3,3)\\ \hline
\nu_0(A_\Lambda)&0&0&0&10&6&6&6&2&18\\
\operatorname{tr}(A_\Lambda^4)&768&816&960&864&768&912&864&912&1536
\end{array}
\]
and these nine values of $I$ are pairwise distinct. Hence graphs corresponding to distinct sets are nonisomorphic. Together with Lemma~$\ref{lem:Q4p-evident-transformations}$, this proves all three parts of Proposition~$\ref{prop:Q4p-classification}$ for $p=3$.

\section*{Acknowledgements}

The author thanks Zhishuo Zhang for helpful consultations on regular coverings and voltage assignments.

\section*{Declaration of generative AI use}

Before using generative AI, the author had conjectured Theorem~\ref{thm:intro-Q4p-4BCI} and had obtained a proof of it except for one gap: the author knew which graphs in the $C$- and $M$-families are isomorphic, but had no proof of this. A literature search assisted by generative AI made the author aware that quotient and lifting arguments based on regular coverings and voltage assignments had been used in other research in algebraic graph theory, although the author had not seen them used explicitly for isomorphism problems. The quotient--lifting method of Section~\ref{sec:reduction}, which is based on these arguments, closed the remaining gap.

The author also used generative AI to assist in writing the Magma programs in Appendix~\ref{apdx:computational-verification} and in improving the presentation and language of the manuscript. After using generative AI, the author reviewed and edited the content, including the Magma programs, as needed and takes full responsibility for the content of the manuscript.


\begin{thebibliography}{99}

\bibitem{ABGR2020}
M.~Arezoomand, A.~Behmaram, M.~Ghasemi, and P.~Raeighasht,
Isomorphisms of bi-Cayley graphs on dihedral groups,
\textit{Discrete Math. Algorithms Appl.} \textbf{12} (2020), Paper No. 2050051.

\bibitem{Babai1977}
L.~Babai,
Isomorphism problem for a class of point-symmetric structures,
\textit{Acta Math. Acad. Sci. Hungar.} \textbf{29} (1977), 329--336.

\bibitem{ConderMa2013}
M.~D.~E.~Conder and J.~Ma,
Arc-transitive abelian regular covers of cubic graphs,
\textit{J. Algebra} \textbf{387} (2013), 215--242.

\bibitem{FengKwak2004}
Y.-Q.~Feng and J.~H.~Kwak,
$s$-Regular cubic graphs as coverings of the complete bipartite graph $K_{3,3}$,
\textit{J. Graph Theory} \textbf{45} (2004), 101--112.

\bibitem{GT1977}
J.~L.~Gross and T.~W.~Tucker,
Generating all graph coverings by permutation voltage assignments,
\textit{Discrete Math.} \textbf{18} (1977), 273--283.

\bibitem{JinLiu2010}
W.~Jin and W.~J.~Liu,
A classification of nonabelian simple $3$-BCI-groups,
\textit{European J. Combin.} \textbf{31} (2010), 1257--1264.

\bibitem{JinLiu2013}
W.~Jin and W.~J.~Liu,
On isomorphisms of small order bi-Cayley graphs,
\textit{Util. Math.} \textbf{92} (2013), 317--327.

\bibitem{KoikeKovacs2014}
H.~Koike and I.~Kov\'acs,
Isomorphic tetravalent cyclic Haar graphs,
\textit{Ars Math. Contemp.} \textbf{7} (2014), 215--235.

\bibitem{KoikeKovacs2019}
H.~Koike and I.~Kov\'acs,
A classification of nilpotent $3$-BCI groups,
\textit{Int. J. Group Theory} \textbf{8} (2019), 11--24.

\bibitem{Li2002}
C.~H.~Li,
On isomorphisms of finite Cayley graphs---a survey,
\textit{Discrete Math.} \textbf{256} (2002), 301--334.

\bibitem{Ma1997}
H.-C.~Ma,
On isomorphisms of Cayley digraphs on dicyclic groups,
\textit{Australas. J. Combin.} \textbf{16} (1997), 189--194.

\bibitem{MMP2004}
A.~Malni\v{c}, D.~Maru\v{s}i\v{c}, and P.~Poto\v{c}nik,
Elementary abelian covers of graphs,
\textit{J. Algebraic Combin.} \textbf{20} (2004), 71--97.

\bibitem{MNS2000}
A.~Malni\v{c}, R.~Nedela, and M.~\v{S}koviera,
Lifting graph automorphisms by voltage assignments,
\textit{European J. Combin.} \textbf{21} (2000), 927--947.

\bibitem{PP2017}
P.~Poto\v{c}nik and R.~Po\v{zar},
Smallest tetravalent half-arc-transitive graphs with the vertex-stabiliser isomorphic to the dihedral group of order $8$,
\textit{J. Combin. Theory Ser. A} \textbf{145} (2017), 172--183.

\bibitem{PS2019}
P.~Poto\v{c}nik and P.~Spiga,
Lifting a prescribed group of automorphisms of graphs,
\textit{Proc. Amer. Math. Soc.} \textbf{147} (2019), 3787--3796.

\bibitem{XFK2022}
J.-H.~Xie, Y.-Q.~Feng, and Y.~S.~Kwon,
Normal Cayley digraphs of generalized quaternion groups with CI-property,
\textit{Appl. Math. Comput.} \textbf{422} (2022), Paper No. 126966.

\bibitem{XieZhang2026}
J.-H.~Xie and Z.~Zhang,
Isomorphisms of bi-Cayley graphs on generalized quaternion groups,
\textit{Comm. Algebra} \textbf{54} (2026), 459--470.

\bibitem{XuJinEtAl2008}
S.~J.~Xu, W.~Jin, Q.~Shi, Y.~Zhu, and J.~J.~Li,
The BCI-property of the bi-Cayley graphs,
\textit{J. Guangxi Norm. Univ. Nat. Sci. Ed.} \textbf{26} (2008), 33--36.

\end{thebibliography}
\end{document}